\documentclass{amsart}
\usepackage{amsmath,amsthm,amssymb,latexsym,epic,bbm,mathbbol,comment, mathrsfs}
\usepackage{tikz-cd}
\usepackage{graphicx,enumerate,stmaryrd,color}
\usepackage[all,2cell]{xy}
\xyoption{2cell}
\usepackage[active]{srcltx}
\usepackage[parfill]{parskip}
\usepackage{enumerate}
\usepackage{url}
\usepackage{hyperref,xparse}
\usepackage{mathtools}

\newtheorem{theorem}{Theorem}
\newtheorem{lemma}[theorem]{Lemma}
\newtheorem{cor}[theorem]{Corollary}
\newtheorem{prop}[theorem]{Proposition}

\newtheorem{proposition}[theorem]{Proposition}
\newtheorem{corollary}[theorem]{Corollary}

\theoremstyle{definition}
\newtheorem{definition}[theorem]{Definition}
\newtheorem{rmk}[theorem]{Remark}

\newtheorem{example}[theorem]{Example}

\font\sc=rsfs10
\newcommand{\csym}[1]{\sc\mbox{#1}\hspace{1.0pt}}

\font\scc=rsfs7
\newcommand{\ccf}[1]{\scc\mbox{#1}\hspace{0.5pt}}

\newcommand{\cccsym}[1]

\newcommand{\on}[1]{\operatorname{#1}}

\newcommand{\setj}[1]{\left\{ #1 \right\}}

\DeclareMathAlphabet\EuRoman{U}{eur}{m}{n}
\SetMathAlphabet\EuRoman{bold}{U}{eur}{b}{n}
\newcommand{\euler}{\EuRoman}

\DeclareMathOperator{\Hom}{\mathrm{Hom}}

\begin{document}
\pagestyle{plain}
\title{Simple transitive 2-representations of bimodules over radical square zero  Nakayama algebras via localization}
\author{Helena Jonsson and Mateusz Stroi\'{n}ski}

\begin{abstract}
We study the classification problem of simple transitive 2-repre\-sen\-ta\-tions of the 2-category of finite-dimensional bimodules over a radical square zero Nakayama algebra. 
This results in a complete classification of simple transitive 2-representations whose apex is a finitary two-sided cell.
We define a notion of localization of 2-representations. 
We construct previously unknown simple transitive 2-representations as localizations of cell 2-representations.
Using the universal property of our construction we prove that any simple transitive 2-representation with finitary apex is equivalent to a localization of a cell 2-representation.
\end{abstract}

\maketitle
\section{Introduction}

The main aim of this paper is to prove a significant generalization of \cite[Conjecture~2]{Jo3}.
The conjectured result concerns the classification of a class of simple transitive $2$-representations of the $2$-category of all bimodules over the algebra $D = \Bbbk[x]/\big(x^{2}\big)$. We prove not only the aforementioned conjecture, but also an analogous statement for the case where the algebra $D$ is replaced by an arbitrary radical square Nakayama algebra.

Ever since the notion of categorification was first formulated in \cite{Cr, CF}, the role of categorical actions in various areas of mathematics has been increasingly important. Among the instances of categorical actions are celebrated results such as the categorification of the Jones polynomial in \cite{Kh} and the proof of Brou{\' e}'s conjecture for symmetric groups in \cite{CR}. The increasing abundance of categorical actions led to the emergence of systematic approaches to higher representation theory, such as those initiated in \cite{CR}, \cite{Os}, and \cite{MM1}.

The program initiated by \cite{MM1} developed the setting of so-called finitary $2$-cate\-go\-ries and their finitary $2$-representations. It provides an axiomatic framework for $2$-representation theory which crucially admits an abstract, well-behaved categorification of a simple module, known as a simple transitive $2$-representation. 
The resulting $2$-representation theory is general enough to include many interesting previously known $2$-representations, especially those coming from Lie theory, such as the actions of projective functors on the BGG category $\mathcal{O}$. At the same time, due to the multiple finiteness assumptions, it is also restrictive enough to allow for classification of simple transitive $2$-representations for various finitary \mbox{$2$-cate\-gories.}

Given a finite-dimensional algebra $A$, a monoidal subcategory $\mathcal{A}$ of $(A\!\on{-mod-}\!A,\otimes_{A})$ which admits an additive generator gives rise to a finitary $2$-category $\csym{A}$, by delooping and strictifying. 
The $2$-category of $2$-representations of $\mathcal{A}$ is biequivalent to the $2$-category of $2$-representations of $\csym{A}$.
An important example of a finitary $2$-category is $\csym{C}_{\!A}$, which is obtained from $(\on{add}\setj{A, A \otimes_{\Bbbk} A}, \otimes_{A})$ via the procedure described above.
One of the earliest complete classification results in finitary $2$-representation theory, \cite[Theorem~15]{MM5}, states that, for self-injective $A$, any simple transitive $2$-representation of $\csym{C}_{\!A}$ is equivalent to a so-called cell $2$-representation. The generalization, \cite[Theorem~12]{MMZ2}, allows us to remove the assumption about $A$ being self-injective, but it is a much more difficult result. This can be explained by the fact that $\csym{C}_{\!A}$ is fiat (in the terminology of \cite{Os, EGNO}, it admits left and right duals) if and only if $A$ is self-injective.

 Most of the finitary $2$-categories that brought the subject its initial interest are fiat. However, from the point of view of $2$-representation theory, the question of removing the assumption about $A$ being self-injective in the example above is natural. Further, it is easy to construct other non-fiat $2$-categories of bimodules over finite-dimensional algebras. The simple transitive $2$-representations of an example of such a $2$-category were studied in \cite{Zi2}, and later classified in \cite{St}. In the fiat case, the main tool for classification of simple transitive $2$-representations is the study of coalgebra $1$-morphisms, as described in \cite{MMMT}. In the non-fiat case, \cite{St} constructs $2$-representations as $2$-categorical weighted colimits of previously known $2$-representations.
 
 In view of the preceding paragraphs, an immediate question to ask is whether it is possible to classify simple transitive $2$-representations of {\it $A\!\on{-mod-}\!A$ itself}, for some well-chosen $A$. This $2$-category is finitary if and only if $A \otimes_{\Bbbk} A^{\on{op}}$ is of finite type. All such algebras belong to a single countable family, see \cite[Theorem~1]{MZ2} for a detailed account.
 In this paper we consider a countable family $\{ \Lambda_n\}$ of algebras such that $\Lambda_n\otimes_\Bbbk \Lambda_n^{\mathrm{op}}$ is of tame type, consisting of radical square  zero Nakayama algebras.

 Our object of study is the family of $2$-categories $\csym{D}_{n}$ associated to $\Lambda_n\!\on{-mod-}\!\Lambda_n$. In particular, these $2$-categories are not finitary, neither do they fit the generalized finitary setting of \cite{Ma1, Ma2}.
 We use and generalize the results of \cite{Jo2,Jo3}. From \cite[Theorem~1]{Jo2}, we know that all except for one of the idempotent $\mathcal{J}$-cells of $\csym{D}_{n}$ are finitary. 
 
 Our main result, Theorem~\ref{thm_main}, gives a complete classification of simple transitive $2$-representations of $\csym{D}_{n}$ with finitary apex, for all $n \geq 1$.
 
 To prove this, we use the approach of \cite{St} and construct the $2$-representations indicated in \cite[Conjecture 2]{Jo3} as $2$-categorical colimits. Interpreted in the $2$-category $\mathbf{Cat}$, the colimit we consider gives the classical notion of localization of categories, as described in \cite{GZ}. Indeed, the $2$-representations we construct can be described as localizations of cell $2$-representations, acting on localizations of the target categories of said cell $2$-representations. In particular, we prove that, generally, the localization of a simple transitive $2$-representation is simple transitive (Proposition~\ref{FinitaryLocalization}). 
 
 In our case, the use of colimits can be argued to be more essential to the classification than in the case of \cite{St}. Indeed, \cite[Theorem~21]{Jo3} shows that the $2$-representations we construct cannot be obtained using coalgebra $1$-morphisms. This illustrates a fundamental difference between the general (locally) finitary $2$-representation theory and the $2$-representation theory of fiat $2$-categories.
 
  We remark that, following \cite{MMMTZ2}, we pass from the $2$-categorical setting to the bicategorical setting. As explained in \cite[Section~2]{MMMTZ2}, the resulting classification problem for simple transitive $2$-representations is the same as that in the strict $2$-categorical setting.

  This paper is organized as follows: in Section~\ref{s2} we introduce the bicategorical setting in which we work and recall necessary definitions and facts regarding birepresentations. 
  In Section~\ref{s3} we introduce localization of birepresentations and prove some elementary results about it.
  In Section~\ref{s4} we introduce and give a detailed description of the bicategory $\csym{D}_{n}$ of finite-dimensional bimodules over a radical square zero Nakayama algebra and state the main theorem.
  Section~\ref{s5} determines the possible action matrices associated to a simple transitive $2$-representation of $\csym{D}_{n}$. 
  Finally, in Section~\ref{s6} we use localization to construct a family of new birepresentations, and prove that any simple transitive birepresentation with finitary apex is equivalent to a member of this family. Thereby we  complete the proof of our main result,
 % this family exhausts simple transitive birepresentations of $\csym{D}_{n}$ with finitary complete the proof of our main result,
 % prove our main result, 
  Theorem~\ref{thm_main}.

\vspace{5mm}
\textbf{Acknowledgements.}
This research is partially supported by Göran Gustafssons Stiftelse. The authors would like to thank their advisor Volodymyr Mazorchuk for many helpful discussions and comments.

\section{Preliminaries on bicategories and birepresentations}\label{s2}

\subsection{Bicategorical setup}
Throughout we assume $\Bbbk$ to be a field. In Section~\ref{s3} this assumption could be relaxed by letting $\Bbbk$ be an integral domain, whereas in Sections~\ref{s4},~\ref{s5},~\ref{s6} we will further assume that $\Bbbk$ is algebraically closed and of characteristic zero.

A bicategory $\csym{C}$ is {\em $\Bbbk$-linear} if, for all $\mathtt{i,j} \in \on{Ob}\csym{C}$, the category $\csym{C}(\mathtt{i,j})$ is $\Bbbk$-linear and if composition of $1$-morphisms is given by $\Bbbk$-bilinear functors. 
Observe that we do not implicitly assume $\Bbbk$-linear categories to be additive or idempotent split - these are properties we treat separately, and, indeed, many of the $\Bbbk$-linear categories we consider (e.g. free $\Bbbk$-linear categories) will not have these properties.
Equivalently, a $\Bbbk$-linear bicategory is a category enriched in the monoidal $2$-category $\mathbf{Cat}_{\Bbbk}$, consisting of $\Bbbk$-linear categories, $\Bbbk$-linear functors and natural transformations. 
For detailed accounts of monoidal bicategories and of enriched bicategories, see \cite{GPS}, \cite{GS}.

 A category $\mathcal{C}$ is {\em finitary} if it is small and equivalent to the category of finitely generated projective modules over a finite-dimensional associative $\Bbbk$-algebra. A $\Bbbk$-linear bicategory $\csym{C}$ is {\em finitary} if, for all $\mathtt{i,j} \in \on{Ob}\csym{C}$, the category $\csym{C}(\mathtt{i,j})$ is finitary.
 
 Given $\Bbbk$-linear bicategories $\csym{C}, \csym{D}$, a $\Bbbk$-linear pseudofunctor $\mathbf{M}: \csym{C} \rightarrow \csym{D}$ is a pseudofunctor of the underlying bicategories such that the functor $\mathbf{M}_{\mathtt{i,j}}$ is $\Bbbk$-linear, for all $\mathtt{i,j} \in \on{Ob}\csym{C}$.

 By $\mathfrak{A}_{\Bbbk}$ we denote the $(1,2)$-full $2$-subcategory of $\mathbf{Cat}_{\Bbbk}$ whose objects are finitary categories. By $\mathfrak{R}_{\Bbbk}$ we denote the $2$-full $2$-subcategory of $\mathbf{Cat}_{\Bbbk}$ whose objects are abelian $\Bbbk$-linear categories and whose $1$-morphisms are right exact functors between such categories. By $\mathbf{Cat}_{\Bbbk}^{\euler{D}}$ we denote the $(1,2)$-full $2$-subcategory of $\mathbf{Cat}_{\Bbbk}$ whose objects are $\Bbbk$-linear, additive and idempotent split categories. Note that such a category is finitary if and only if it is hom-finite and has finitely many isomorphism classes of indecomposable objects.
 
 For a $\Bbbk$-linear bicategory $\csym{C}$, we refer to $\Bbbk$-linear pseudofunctors $\csym{C} \rightarrow \mathbf{Cat}_{\Bbbk}$ as {\it birepresentations of $\csym{C}$}.
 In particular, a {\it finitary birepresentation} is a $\Bbbk$-linear pseudofunctor $\csym{C} \rightarrow \mathfrak{A}_{\Bbbk}$. Similarly, an {\it abelian birepresentation of $\csym{C}$} is a $\Bbbk$-linear pseudofunctor $\csym{C} \rightarrow \mathfrak{R}_{\Bbbk}$. We will mainly be focusing on finitary birepresentations.
 
 Given $\Bbbk$-linear bicategories $\csym{C},\csym{D}$, we denote by $[\csym{C},\csym{D}]$ the $\Bbbk$-linear bicategory of $\Bbbk$-linear pseudofunctors from $\csym{C}$ to $\csym{D}$, strong transformations between such pseudofunctors, and modifications of such transformations. Similarly to $\Bbbk$-linear categories and functors, the suitable ``linearizations'' of standard bicategorical notions, as presented for instance in \cite{Lei}, coincide with the enriched notions given in \cite{GS}.
 
 We remark that, in view of the strictification results of \cite[Section~4.2]{Po}, \cite[Section~2.3]{MMMTZ2}, classification problems for $2$-representations (in the sense of \cite{MM1}) of a $\Bbbk$-linear $2$-category $\csym{C}$ are equivalent to the same classification problems concerning birepresentations of $\csym{C}$, or any bicategory biequivalent to $\csym{C}$.
 
 A finitary birepresentation $\mathbf{M}$ of $\csym{C}$ is called {\em simple transitive} if it does not admit a non-trivial $\csym{C}$-stable ideal, i.e. a family of ideals $\mathcal{I} = (\mathcal{I}(\mathtt{i}) \subseteq \mathbf{M}(\mathtt{i}))_{\mathtt{i} \in \on{Ob}\ccf{C}}$ such that $\mathbf{M}\mathrm{F}(\mathcal{I}(\mathtt{i})) \subseteq \mathcal{I}(\mathtt{j})$, for all $\mathrm{F} \in \csym{C}(\mathtt{i,j})$.
 
 If $\csym{C}$ has a unique object $\mathtt{i}$, the {\em rank} of $\mathbf{M}$ is the number of isomorphism classes of indecomposable objects of $\mathbf{M}(\mathtt{i})$. If $\csym{C}$ has more objects, one may consider a function $\on{rank}_{\mathbf{M}}: \on{Ob}\csym{C} \rightarrow \mathbb{N}$.
 
 With the exception of Section~\ref{s3}, we will let $\csym{C}$ be a finitary bicategory and $\mathbf{M}$ a finitary birepresentation of $\csym{C}$, unless otherwise stated.

 \subsection{Cells}
 
 The {\em left preorder} $\leq_{L}$ on the set of isomorphism classes of indecomposable $1$-morphisms of $\csym{C}$ is defined by setting $F \leq_{L} G$ if there is a $1$-morphism $H$ such that $G$ is a direct summand of $H \circ F$. We denote the resulting equivalence relation by $\sim_{L}$, and refer to its equivalence classes as {\em left cells.} Similarly one defines the right and two-sided preorders $\leq_{R}, \leq_{J}$, together with right and two-sided equivalence relations and right and two-sided cells.
 
 Let $\mathbf{M}$ be a simple transitive birepresentation of $\csym{C}$. By \cite[Lemma 1]{CM}, the collection of two-sided cells of $\csym{C}$ which are not annihilated by $\mathbf{M}$ admits a unique maximal element $\mathcal{J}$ with respect to the two-sided order. We refer to $\mathcal{J}$ as the {\it apex} of $\mathbf{M}$.
    
 A two-sided cell $\mathcal{J}$ is called \emph{idempotent} given that there exist $F,G,H\in \mathcal{J}$ such that $F$ is a direct summand of $G\circ H$. The apex is necessarily idempotent, see \cite[Lemma~1]{CM}.
 
 Let $\mathcal{L}$ be a left cell of $\csym{C}$. There is then a unique object $\mathtt{i}$ of $\csym{C}$ which is the domain of all $1$-morphisms in $\mathcal{L}$. To $\mathcal{L}$ we associate a simple transitive subquotient $\mathbf{C}_{\mathcal{L}}$ of the principal birepresentation $\mathbf{P}_{\mathtt{i}}=\csym{C}(\mathtt{i},-)$; see \cite[Subsection~3.3]{MM5} for details. We call 
 $\mathbf{C}_{\mathcal{L}}$ is called the \emph{cell birepresentation} corresponding to $\mathcal{L}$.
 
 \subsection{Action matrices}
 
 Let $\mathcal{C,D}$ be a pair of finitary categories and let $F: \mathcal{C} \rightarrow \mathcal{D}$ be a $\Bbbk$-linear functor. Let $X_{1},\ldots, X_{n}$ be a complete, irredundant list of isomorphism classes of indecomposable objects in $\mathcal{C}$ and let $Y_{1},\ldots, Y_{m}$ be such a list for $\mathcal{D}$. With respect to these, the {\em action matrix} $[F]$ of $F$ is the $m\times n$ matrix with non-negative integer entries, defined by
 \[
  [F]_{ij} =  \text{multiplicity of }Y_{i} \text{ as a direct summand of } FX_{j} .
 \]
 
 In particular, given a finitary birepresentation $\mathbf{M}$ of $\csym{C}$ and a $1$-morphism $F$ of $\csym{C}$, the functor $\mathbf{M}F$ satisfies the above assumptions. Hence we obtain an action matrix $[\mathbf{M}F]$. If there is no risk of ambiguity, we may sometimes write $[F]$ for $[\mathbf{M}F]$.

 \subsection{Abelianization}
 Given a finitary birepresentation $\mathbf{M}$ of $\mathscr{C}$, we have its \emph{(projective) abelianization} $\overline{\mathbf{M}}$ as defined  in \cite[Section~3]{MMMT}. 
 It is a pseudofunctor $\mathscr{C}\to \mathfrak{R}_\Bbbk$, so that $\overline{\mathbf{M}}$ is an abelian birepresentation of $\mathscr{C}$. 
 Up to equivalence, $\mathbf{M}$ is recovered by restricting to the subcategories of projective objects in the underlying (abelian) categories of $\overline{\mathbf{M}}$.
 
 \subsection{Additive and Karoubi envelopes}~\\
 Let $\mathbf{Cat}_{\Bbbk}^{\oplus}$ denote the $(1,2)$-full $2$-subcategory of $\mathbf{Cat}_{\Bbbk}$ whose objects are additive $\Bbbk$-linear categories. Similarly, let $\mathbf{Cat}_{\Bbbk}^{\mathcal{K}}$ denote the $(1,2)$-full $2$-subcategory of $\mathbf{Cat}_{\Bbbk}$ whose objects are idempotent split $\Bbbk$-linear categories. The respective inclusion $2$-functors $\mathbf{Cat}_{\Bbbk}^{\oplus} \hookrightarrow \mathbf{Cat}_{\Bbbk}$ and $\mathbf{Cat}_{\Bbbk}^{\mathcal{K}} \hookrightarrow \mathbf{Cat}_{\Bbbk}$ admit bicategorically left adjoint $2$-functors $(-)^{\oplus}, (-)^{\mathcal{K}}$, known as the {\it additive} and {\it Karoubi} envelopes, respectively. The Karoubi envelope restricts to a bicategorical left adjoint to the inclusion $\mathbf{Cat}_{\Bbbk}^{\euler{D}} \hookrightarrow \mathbf{Cat}_{\Bbbk}^{\oplus}$. Thus the composition $(-)^{\euler{D}} := (-)^{\on{Kar}} \circ (-)^{\oplus}$ is a bicategorical left adjoint to the inclusion $\mathbf{Cat}_{\Bbbk}^{\euler{D}} \hookrightarrow \mathbf{Cat}_{\Bbbk}$. For more detailed accounts of the envelope constructions and of the bicategorical adjunction given above, see \cite{Ri} and \cite[Section~3]{St}.
 
\section{Localization of birepresentations}\label{s3}

\subsection{Bicategorical weighted colimits and localization}
 We now give a brief recollection of the general treatment of bicategorical weighted colimits of birepresentations given in \cite{St}.

Given a small $\Bbbk$-linear bicategory $\csym{I}$, a $\Bbbk$-linear bicategory $\csym{B}$, a $\Bbbk$-linear pseudofunctor $\mathbf{F}: \csym{I} \rightarrow \csym{B}$ and a $\Bbbk$-linear pseudofunctor $\mathbf{W}: \csym{I}^{\on{op}} \rightarrow \mathbf{Cat}_{\Bbbk}$, a {\it $\mathbf{W}$-weighted $\Bbbk$-linear bicategorical colimit of $\mathbf{F}$} is an object $\mathbf{W}\star \mathbf{F}$ of $\csym{B}$ together with $\Bbbk$-linear equivalences of categories
\begin{equation}\label{ColimitEquation}
 \csym{B}(\mathbf{W}\star \mathbf{F}, \mathtt{b}) \simeq [\csym{I}^{\on{op}},\mathbf{Cat}_{\Bbbk}](\mathbf{W},\csym{B}(\mathbf{F}-,\mathtt{b})) \text{, for } \mathtt{b} \in \on{Ob}\csym{B},
\end{equation}
strongly natural in $\mathtt{b}$. Combining various results of \cite{GS}, \cite{Ke1}, \cite{Ke2} and \cite{Ri}, the main conclusions of \cite[Section~3]{St} are the following: 
\begin{itemize}
 \item if $\csym{B} = [\csym{C},\mathbf{Cat}_{\Bbbk}^{\euler{D}}]$, then $\mathbf{W} \star \mathbf{F}$ exists for all choices of $\mathbf{W},\mathbf{F}$;
 \item weighted colimits in $[\csym{C},\mathbf{Cat}_{\Bbbk}^{\euler{D}}]$ can be computed pointwise in $\mathbf{Cat}_{\Bbbk}^{\euler{D}}$;
 \item the pointwise computation can be facilitated by using the Karoubi and additive envelope $2$-functors, which preserve bicategorical colimits.
\end{itemize}

One of the earliest studied bicategorical colimits is localization of categories, as described in \cite{GZ}. In the setup described above, we may formulate it as follows:
\begin{itemize}
 \item let $\csym{I}$ be the $2$-category
$\begin{tikzcd}[ampersand replacement = \&, sep = small]
   \mathtt{i} \arrow[r, shift left=6pt, ""{name=U, below}, "\mathrm{S}"] \arrow[r, shift right=6pt, swap, ""{name=D, above}, "\mathrm{T}"] \& \mathtt{j}
   \arrow[Rightarrow, from=U, to=D, swap, start anchor ={[yshift =3pt]}, end anchor = {[yshift=-3.8pt]}]
  \end{tikzcd}$, with two objects, two parallel non-identity $1$-morphisms, and a unique non-identity $2$-morphism between these. 
  \item Let $\mathbf{W}: \csym{I}^{\on{op}} \rightarrow \mathbf{Cat}$ be the $2$-functor depicted by
  $\begin{tikzcd}[ampersand replacement = \&, sep = small]
   \Big({\scriptscriptstyle \euler{s}} \arrow[r, shift left, "{\scriptscriptstyle f}"] \&[-0.6em] {\scriptscriptstyle \euler{t}}\Big) 
   \arrow[l, shift left, "{\scriptscriptstyle f^{-1}}"]
   \arrow[from=r, shift right=6pt, swap, "{\scriptscriptstyle \euler{s}}", ""{name=U, below}] \arrow[from=r, shift left=6pt, "{\scriptscriptstyle \euler{t}}" , ""{name=D, above}] \& \euler{1}
   \arrow[Rightarrow, from=U, to=D, "{\scriptscriptstyle f}", start anchor ={[yshift =3pt]}, end anchor = {[yshift=-3.8pt]}]
  \end{tikzcd}$,
  sending
  \begin{itemize}
      \item $\mathtt{j}$ to the terminal category $\euler{1}$,
      \item $\mathtt{i}$ to the walking isomorphism category,
      \item $\mathrm{S}$ to the functor choosing the domain of the walking isomorphism $f$,
      \item $\mathrm{T}$ to the codomain of $f$,
      \item the unique non-identity $2$-morphism to the natural transformation given by $f$.
  \end{itemize}
  \item Let $\mathcal{C}$ be a category and let $\euler{2}$ be the walking arrow category, $\euler{2} = 1 \xrightarrow{\omega} 2$. Consider the arrow category $\mathcal{C}^{\rightarrow} := \mathbf{Cat}(\euler{2}, \mathcal{C})$. The functors 
  \[
   \on{dom}: \euler{1} \xrightarrow{1 \mapsto 1} \euler{2} \text{ and } \on{cod}: \euler{1} \xrightarrow{1 \mapsto 2} \euler{2}
  \]
   from the terminal category to $\euler{2}$ induce functors 
   \[
   \mathcal{C}^{\rightarrow} \xrightarrow{\on{dom}_{\mathcal{C}}} \mathcal{C} \text{ and } \mathcal{C}^{\rightarrow} \xrightarrow{\on{cod}_{\mathcal{C}}} \mathcal{C},
   \]
   respectively. Further, the natural transformation $\omega: \on{dom} \Rightarrow \on{cod}$ gives rise to the natural transformation $\omega_{\mathcal{C}}: \on{dom}_{\mathcal{C}} \Rightarrow \on{cod}_{\mathcal{C}}$. 
  \item 
  Let $S$ be a collection of morphisms of $\mathcal{C}$; equivalently, $S$ gives a collection of objects of $\mathcal{C}^{\rightarrow}$. Let $\mathcal{S}$ be the full subcategory of $\mathcal{C}^{\rightarrow}$ satisfying $\on{Ob}\mathcal{S} = S$, and let $\mathrm{I}: \mathcal{S} \rightarrow \mathcal{C}^{\rightarrow}$ be its inclusion functor.
  \item 
  We define a $2$-functor $\mathbf{F}: \csym{I} \rightarrow \mathbf{Cat}$ by the diagram
  $\begin{tikzcd}[ampersand replacement = \&, sep = huge]
   \mathcal{S} \arrow[r, shift left=6pt, ""{name=U, below}, "\on{dom}_{\mathcal{C}}\circ \mathrm{I}"] \arrow[r, shift right=6pt, swap, ""{name=D, above}, "\on{cod}_{\mathcal{C}}\circ \mathrm{I}"] \& \mathcal{C}
   \arrow[Rightarrow, from=U, to=D, start anchor ={[yshift =3pt]}, end anchor = {[yshift=-3.8pt]}, "\omega_{\mathcal{C}}\bullet \mathrm{I}"]
  \end{tikzcd}$.
  \item 
 The bicategorical colimit $\mathbf{W}\star \mathbf{F}$, known as the coinverter of $\mathbf{F}$, is the localization $\mathcal{C}[S^{-1}]$. 
\end{itemize}
  
  We now describe its universal property, using the description of bicategorical colimits above (or e.g. the description of coinverters in \cite[Section~6.6]{La}).
  Given a category $\mathcal{D}$, let $\mathbf{Cat}(\mathcal{C},\mathcal{D})^{S}$ be the full subcategory of $\mathbf{Cat}(\mathcal{C},\mathcal{D})$, an object of which is a functor $\mathrm{F}$ such that $\mathrm{F} \bullet \omega_{\mathcal{C}} \bullet \mathrm{I}$ is a natural isomorphism. Equivalently, $\mathrm{F}(s)$ is an isomorphism in $\mathcal{D}$, for any $s \in S$. The universal property is given by an equivalence of $2$-functors:
  \[
   \mathbf{Cat}(\mathcal{C}[S^{-1}],-) \xrightarrow{\euler{Q}} \mathbf{Cat}(\mathcal{C},-)^{S}.
  \]
  Using the bicategorical Yoneda lemma, we may also obtain the localization functor $\mathrm{Q} \in \mathbf{Cat}(\mathcal{C},\mathcal{C}[S^{-1}])$, which satisfies $\euler{Q} = \mathbf{Cat}(\mathrm{Q},-)$.

  Given a $\Bbbk$-linear, additive, idempotent split category $\mathcal{C}$ and a small category $\mathcal{I}$, the functor category $\mathbf{Cat}(\mathcal{I},\mathcal{C})$ also is $\Bbbk$-linear, additive and idempotent split. In fact, we have the canonical isomorphism $\mathbf{Cat}(\mathcal{I},\mathcal{C}) \simeq \mathbf{Cat}_{\Bbbk}(\Bbbk\mathcal{I},\mathcal{C})$, where $\Bbbk \mathcal{I}$ is the free $\Bbbk$-linear category on $\mathcal{I}$.
  We thus obtain a $\Bbbk$-linear $2$-functor 
  \[
  \mathbf{Cat}(\mathcal{I},-): \mathbf{Cat}_{\Bbbk}^{\euler{D}} \rightarrow \mathbf{Cat}_{\Bbbk}^{\euler{D}}.
  \]
  Consider again the particular case $\mathcal{I} = \euler{2}$. The functors $\on{dom},\on{cod}$ induce $\Bbbk$-linear $2$-transformations $\on{Dom}, \on{Cod}: \mathbf{Cat}(2,-) \rightarrow \mathbb{1}_{\mathbf{Cat}_{\Bbbk}^{\euler{D}}}$ and the natural transformation $\omega$ gives a modification $\euler{w}: \on{Dom} \rightarrow \on{Cod}$. 
  
  Let $\mathbf{M}: \csym{C} \rightarrow \mathbf{Cat}_{\Bbbk}^{\euler{D}}$ be a birepresentation of $\csym{C}$ and let $\mathbf{M}^{\rightarrow} := \mathbf{Cat}(\euler{2},-) \circ \mathbf{M}$.
  
  \begin{definition}
    A tuple $\mathcal{S} = (\mathcal{S}(\mathtt{i}))_{\mathtt{i} \in \on{Ob}\ccf{C}}$, where $\mathcal{S}(\mathtt{i})$ is a collection of morphisms of $\mathbf{M}(\mathtt{i})$, is said to be a {\it $\csym{C}$-stable collection $\mathcal{S}$ in $\mathbf{M}$} if, for any $\mathtt{i,j} \in \on{Ob}\csym{C}$ and any $\mathrm{F} \in \csym{C}(\mathtt{i,j})$, we have
   \[
    \mathbf{M}\mathrm{F}\left( \mathcal{S}(\mathtt{i}) \right) \subseteq \mathcal{S}(\mathtt{j}).
   \]
   We say that $\mathcal{S}$ is {\it multiplicative} if $\mathcal{S}(\mathtt{i})$ is a subcategory of $\mathbf{M}(\mathtt{i})$, for all $\mathtt{i} \in \on{Ob}\csym{C}$.
  \end{definition}

   As an immediate consequence of the definition, there is a canonical correspondence between locally full subbirepresentations of $\mathbf{M}^{\rightarrow}$ and $\csym{C}$-stable collections in $\mathbf{M}$. Indeed, a locally full subbirepresentation $\mathbf{K}$ of a birepresentation $\mathbf{N}$ of $\csym{C}$ is uniquely determined by a tuple $(\mathcal{K}(\mathtt{i}))_{\mathtt{i} \in \on{Ob}\ccf{C}}$ of collections of objects of $\mathbf{N}(\mathtt{i})$ such that $\mathbf{N}\mathrm{F}(\mathcal{K}(\mathtt{i})) \subseteq \mathcal{K}(\mathtt{j})$.

   Let $\mathbf{S} \xrightarrow{\euler{I}} \mathbf{M}^{\rightarrow}$ be a locally full subbirepresentation of $\mathbf{M}^{\rightarrow}$, and let $\mathcal{S}$ be its corresponding $\csym{C}$-stable collection.
  
   Consider again the $2$-category $\csym{I}$ which we used in the above definition of localization of categories. Let $\Bbbk\!\csym{I}$ be the free $\Bbbk$-linear $2$-category on $\csym{I}$ (its objects and $1$-morphisms coincide with those of $\csym{I}$, and its spaces of $2$-morphisms are linearizations of the sets of $2$-morphisms in $\csym{I}$). A $\Bbbk$-linear pseudofunctor from $\Bbbk\!\csym{I}$ to a $\Bbbk$-linear bicategory $\csym{D}$ can be canonically identified with an ordinary pseudofunctor from $\csym{I}$ to (the underlying bicategory of) $\csym{D}$. Hence, if $\csym{D}$ is a $\Bbbk$-linear $2$-category, then a $\Bbbk$-linear $2$-functor $\Bbbk\!\csym{I} \rightarrow \csym{D}$ is given by a diagram in $\csym{D}$ of shape $\csym{I}$.
   
   In particular, we have the diagram 
  $\begin{tikzcd}[ampersand replacement = \&, sep = huge]
   \mathbf{Cat}(\euler{2},-) \arrow[r, shift left=6pt, ""{name=U, below}, "\on{Dom}"] \arrow[r, shift right=6pt, swap, ""{name=D, above}, "\on{Cod}"] \& \mathbb{1}_{\mathbf{Cat}_{\Bbbk}^{\euler{D}}}
   \arrow[Rightarrow, from=U, to=D, start anchor ={[yshift =3pt]}, end anchor = {[yshift=-3.8pt]}, "\euler{w}"]
  \end{tikzcd}$
  which we may precompose with $\mathbf{M}$ to obtain the diagram
  $\begin{tikzcd}[ampersand replacement = \&, sep = huge]
   \mathbf{M}^{\rightarrow} \arrow[r, shift left=6pt, ""{name=U, below}, "\on{Dom}_{\mathbf{M}}"] \arrow[r, shift right=6pt, swap, ""{name=D, above}, "\on{Cod}_{\mathbf{M}}"] \& \mathbf{M},
   \arrow[Rightarrow, from=U, to=D, start anchor ={[yshift =3pt]}, end anchor = {[yshift=-3.8pt]}, "\euler{w}_{\mathbf{M}}"]
  \end{tikzcd}$ from which we form the diagram
  \begin{equation}\label{CoinverterDiagram}
  \begin{tikzcd}[ampersand replacement = \&, sep = huge]
   \mathbf{S} \arrow[r, shift left=6pt, ""{name=U, below}, "\on{Dom}_{\mathbf{M}} \circ \euler{I}"] \arrow[r, shift right=6pt, swap, ""{name=D, above}, "\on{Cod}_{\mathbf{M}} \circ \euler{I}"] \& \mathbf{M}.
   \arrow[Rightarrow, from=U, to=D, start anchor ={[yshift =3pt]}, end anchor = {[yshift=-3.8pt]}, "\euler{w}_{\mathbf{M}} \bullet \euler{I}"]
  \end{tikzcd}
  \end{equation}
 
  \begin{definition}
   Let $\mathbf{M}$ be a birepresentation of $\csym{C}$, let $\mathcal{S}$ be a multiplicative $\csym{C}$-stable collection in $\mathbf{M}$ and let $\mathbf{S}$ be its corresponding locally full subbirepresentation of $\mathbf{M}^{\rightarrow}$. 
  We define the \emph{localization} $\mathbf{M} \rightarrow \mathbf{M}[\mathcal{S}^{-1}]$ of $\mathbf{M}$ by $\mathcal{S}$ as the coinverter of Diagram \ref{CoinverterDiagram} in $[\csym{C},\mathbf{Cat}_{\Bbbk}^{\euler{D}}]$.
  \end{definition}
  Since bicategorical colimits in $[\csym{C},\mathbf{Cat}_{\Bbbk}^{\euler{D}}]$ are constructed pointwise in $\mathbf{Cat}_{\Bbbk}^{\euler{D}}$, we conclude that, for every $\mathtt{i} \in \on{Ob}\csym{C}$, we have $\mathbf{M}[\mathcal{S}^{-1}](\mathtt{i}) \simeq \mathbf{M}(\mathtt{i})[\mathcal{S}(\mathtt{i})^{-1}]$.
  
  Reading off the universal property from \eqref{ColimitEquation}, we find the following:
  \begin{proposition}
  For a birepresentation $\mathbf{N}$, the category $[\csym{C},\mathbf{Cat}_{\Bbbk}^{\euler{D}}](\mathbf{M}[\mathcal{S}^{-1}],\mathbf{N})$ is equivalent to the full subcategory of $[\csym{C},\mathbf{Cat}_{\Bbbk}^{\euler{D}}](\mathbf{M},\mathbf{N})$ whose objects are $\Bbbk$-linear strong transformations $\Theta$ such that $\Theta \bullet \euler{w}_{\mathbf{M}} \bullet \euler{I}$ is an invertible modification. 
  
  Equivalently, $\Theta$ is a strong transformation such that $\Theta_{\mathtt{i}}(s)$ is invertible, for any $\mathtt{i} \in \on{Ob}\csym{C}$ and any $s \in \mathcal{S}(\mathtt{i})$. 
  
  By Yoneda lemma for bicategories, the components  $\Upsilon_{\mathtt{i}}:\mathbf{M}(\mathtt{i}) \rightarrow \mathbf{M}(\mathtt{i})[\mathcal{S}(\mathtt{i})^{-1}]$ of the localization transformation $\Upsilon: \mathbf{M} \rightarrow \mathbf{M}[\mathcal{S}^{-1}]$ are given by the indicated localization functors.
  \end{proposition}

    \begin{example}\label{LocalizeA2}
  Consider the quiver $A_{2}: 1 \xrightarrow{a} 2$. Let $\Bbbk A_{2}$ be the free $\Bbbk$-linear category on $A_{2}$. The $\Bbbk$-linear localization of $\Bbbk A_{2}$ by the morphism $a$ is the free $\Bbbk$-linear category $\Bbbk \widehat{A}_{2}$ on the category $\widehat{A}_{2}$, which admits the following presentation:
  \[
   \widehat{A}_{2} =
   \Big[\begin{tikzcd}
    1 \arrow[r, bend left, "a"] & 2 \arrow[l, bend left, "a^{-1}"]
   \end{tikzcd}\Big] / \left\langle a^{-1}\circ a = \on{id}_{1}, \; a\circ a^{-1} = \on{id}_{2} \right\rangle.
  \]
  Indeed, it is easy to verify that if $\Bbbk A_{2} \xrightarrow{\mathrm{F}} \mathcal{D}$ is a functor such that $\mathrm{F}(a)$ is invertible, then we may uniquely extend $\mathrm{F}$ to $\widehat{\mathrm{F}}: \Bbbk \widehat{A}_{2} \rightarrow \mathcal{D}$ by setting  $\widehat{\mathrm{F}}(a^{-1}) = \mathrm{F}(a)^{-1}$. It is also clear that $\Bbbk\widehat{A}_{2}$ is equivalent to $\Bbbk A_{1}$, where $A_{1}$ is the quiver with a unique vertex and no arrows.

  Hence, 
  \[
 (\Bbbk A_{2}\!\on{-proj})[{\setj{a}^{-1}}] \simeq (\Bbbk A_{2})^{\euler{D}}[\setj{a}^{-1}] \simeq (\Bbbk A_{2})[\setj{a}^{-1}]^{\euler{D}} \simeq (\Bbbk A_{1})^{\euler{D}} \simeq \mathbf{vect}_{\Bbbk}.
  \]
  Similarly, we may extend the above argument to the case of a quiver of the form
  \[
   \begin{tikzcd}[sep = small]
    1 \arrow[d, "s_{1}"] & 2 \arrow[d, "s_{2}"] & \cdots & m \arrow[d, "s_{m}"] \\
    1' & 2' & \cdots & m'
   \end{tikzcd}.
  \]
  
  For $I \subseteq \setj{1,\ldots,m}$, localizing the category of projectives over the path algebra of this quiver by $\setj{s_{i} \; | \; i \in I}$ gives the category of projectives over the path algebra of the quiver obtained by contracting the arrows $\setj{s_{i} \; | \; i \in I}$ and replacing the connected component $i \xrightarrow{s_{i}} i'$ by a single vertex $s_{i}$.
 \end{example}

\subsection{Localization and finitary birepresentations}
We now give two properties of localization of birepresentations which are particularly relevant for the study of finitary birepresentations.

\begin{proposition}\label{FinitaryLocalization}
 Let $\mathbf{M}$ be a finitary birepresentation of $\csym{C}$ and let $\mathcal{S}$ be a $\csym{C}$-stable collection in $\mathbf{M}$. If $\mathbf{M}$ is simple transitive, then so is $\mathbf{M}[\mathcal{S}^{-1}]$.
\end{proposition}

\begin{proof}
 Let $\mathcal{I}$ be an ideal in $\mathbf{M}[\mathcal{S}^{-1}]$, and consider the canonical strong transformation $\mathbf{M}[\mathcal{S}^{-1}] \xrightarrow{\pi} \mathbf{M}[\mathcal{S}^{-1}]/\mathcal{I}$. There is a strong transformation $\overline{\pi}: \mathbf{M} \rightarrow \mathbf{M}[\mathcal{S}^{-1}]/\mathcal{I}$ which sends $\mathcal{S}$ to isomorphisms and makes the following diagram commute up to invertible modification:
 \[
  \begin{tikzcd}
   \mathbf{M} \arrow[r, "\euler{Q}"] \arrow[dr, swap, "\overline{\pi}"] & \mathbf{M}[\mathcal{S}^{-1}] \arrow[d, "\pi"] \\
   & \mathbf{M}[\mathcal{S}^{-1}]/\mathcal{I}
  \end{tikzcd}
 \]
 Consider the ideal $\on{Ker}\overline{\pi}$ of $\mathbf{M}$. Since $\mathbf{M}$ is simple transitive, this ideal is zero or all of $\mathbf{M}$. 
 
 If $\on{Ker} \overline{\pi} = \mathbf{M}$, then $\overline{\pi} = 0$ and hence, since $\overline{\pi}$ determines $\pi$ up to invertible modification, we see that also $\pi = 0$, showing that $\mathcal{I}$ is all of $\mathbf{M}[\mathcal{S}^{-1}]$.
 
 If $\on{Ker}\overline{\pi} = 0$, then $\overline{\pi}$ is locally faithful. But $\overline{\pi}$ being locally faithful implies that also $\euler{Q}$ is locally faithful. However, since $\mathbf{M}$ is finitary, the category $\mathbf{M}(\mathtt{i})$ is balanced (mono and epi implies iso), for all $\mathtt{i} \in \on{Ob}\csym{C}$. A faithful functor from a balanced category reflects isomorphisms, since faithful functors reflect monomorphisms and epimorphisms. We thus see that $\euler{Q}$ is given by conservative functors, which implies that $\mathcal{S}$ consists of isomorphisms, and so $\mathbf{M} \simeq \mathbf{M}[\mathcal{S}^{-1}]$, proving that $\mathcal{I}$ is zero. Beyond the finitary case, the same argument holds whenever $\mathbf{M}(\mathtt{i})$ is balanced, for all $\mathtt{i} \in \on{Ob}\csym{C}$.
\end{proof}

Recall that a $\Bbbk$-linear category is finitary if it is hom-finite, additive, idempotent split and with finitely many isomorphism classes of indecomposable objects. We now show that once we have established hom-finiteness, the last condition follows automatically:

\begin{proposition}
 Let $\mathbf{M}$ be a finitary birepresentation of $\csym{C}$, and let $\mathcal{S}$ be a $\csym{C}$-stable collection in $\mathbf{M}$. If $\mathbf{M}[\mathcal{S}^{-1}]$ is locally hom-finite, then $\mathbf{M}[\mathcal{S}^{-1}]$ is a finitary birepresentation of $\csym{C}$.
\end{proposition}

\begin{proof}
 Under the assumption above, $\mathbf{M}[\mathcal{S}^{-1}](\mathtt{i})$ is hom-finite, additive and idempotent split, for every $\mathtt{i} \in \on{Ob}\csym{C}$. As a consequence, each such category is Krull-Schmidt. For each $\mathtt{i}$, fix a complete list of representatives $X_{1}^{\mathtt{i}},\ldots, X_{m(\mathtt{i})}^{\mathtt{i}}$ of the isomorphism classes of indecomposable objects in $\mathbf{M}(\mathtt{i})$. Since $\mathbf{M} \xrightarrow{\euler{Q}} \mathbf{M}[\mathcal{S}^{-1}]$ is locally essentially surjective, we see that 
 \[
  \mathbf{M}[\mathcal{S}^{-1}](\mathtt{i}) = \on{add}\setj{ \euler{Q}_{\mathtt{i}}(X_{j}^{\mathtt{i}}) \; | \; j = 1,\ldots,m(\mathtt{i})}.
 \]
 Since this category is Krull-Schmidt, we can decompose each of the objects $\euler{Q}_{\mathtt{i}}(X_{j}^{\mathtt{i}})$ into a direct sum of finitely many indecomposable objects. We may thus write
 \[
  \euler{Q}_{\mathtt{i}}(X_{j}^{\mathtt{i}}) = \bigoplus_{k=1}^{n(j)} Y_{j,k}^{\mathtt{i}},
 \]
 and so it follows that
 \[
    \mathbf{M}[\mathcal{S}^{-1}](\mathtt{i}) = \on{add}\setj{ Y_{j,k}^{\mathtt{i}} \; | \; k = 1,\ldots, n(j) \text{ and } j = 1,\ldots,m(\mathtt{i})}.
 \]
 From this we see that indeed there are only finitely many isomorphism classes of indecomposable objects in $\mathbf{M}[\mathcal{S}^{-1}](\mathtt{i})$ and the result follows.
\end{proof}

\section{\texorpdfstring{The bicategory of finite-dimensional $\Lambda_n$-$\Lambda_n$-bimodules and the main result}{The bicategory of finite dimensional A-A-bimodules and the main result}}\label{s4}
From now on we assume $\Bbbk$ to be an algebraically closed field of characteristic 0. 
Further, we assume all modules to be finite-dimensional.
Throughout we use the notions of $\Lambda$-$\Lambda$-bimodules and  left $\Lambda\otimes_\Bbbk \Lambda^\mathrm{op}$-modules interchangeably.

\subsection{\texorpdfstring{The algebra $\Lambda_n$}{The algebra An}}
Let $\Lambda_1$ be the path algebra of the quiver
\begin{displaymath}
\begin{tikzcd}
 1 \arrow[loop right, "\alpha", in=40, out = -20, swap, looseness=5]
\end{tikzcd}
\end{displaymath}
modulo the relation $\alpha^2=0$. Then $\Lambda_1$ is isomorphic to the algebra of dual numbers $D=\Bbbk [x]/(x^2)$.

For $n\geq 2$, let $Q_n$ be the  following quiver:
\[
 \begin{tikzcd}
  & & 1 \arrow[dll, swap, "\alpha_{1}"] \\
  2 \arrow[r, swap, "\alpha_{2}"] & 3 \arrow[r, swap, "\alpha_{3}"] & \cdots \arrow[r, swap, "\alpha_{n-2}"] & n-1 \arrow[r, swap, "\alpha_{n-1}"] & n \arrow[ull, swap, "\alpha_{n}"]
 \end{tikzcd}
\]
Let $\Lambda_n$ be the path algebra $\Bbbk Q_n$ modulo the ideal generated by the relations that composition of  any two arrows is 0.

We denote the orthogonal, primitive idempotents associated to the vertices of $Q_{n}$ by $\varepsilon_{1},\varepsilon_{2},\ldots, \varepsilon_{n}$.

Given a positive integer $n$, let $\mathscr{D}_{n}$ be the bicategory which has a unique object $\mathtt{i}$ such that $\mathscr{D}_n(\mathtt{i},\mathtt{i})=\Lambda_n\text{-mod-}\Lambda_n$, where the composition of $1$-morphisms is given by tensoring over $\Lambda_{n}$.

\subsection{\texorpdfstring{Indecomposable $\Lambda_n$-$\Lambda_n$-bimodules}{Indecomposable An-An-bimodules}}
We now give a brief summary of \cite[Sections~1-2]{Jo2}.

For each $n\geq 1$, the algebra $\Lambda_n\otimes_{\Bbbk} \Lambda_n^\mathrm{op}$ is special biserial in the sense of \cite{BR}. The isomorphism classes of indecomposable finite-dimensional modules over special biserial algebras were classified in \cite{BR},\cite{WW}. We use the notation from \cite{Jo2} (up to a small change of indexing, see Remark~\ref{rmk_index_nm}). 
The case $n=1$ requires slightly different notation, so we do not describe it in detail here but refer the reader to \cite{Jo3}. All statements still hold for $n=1$.

For $n\geq 2$, the algebra $\Lambda_n\otimes_\Bbbk \Lambda_n^\mathrm{op}$ is isomorphic to the path algebra of the discrete torus
\[
\begin{tikzcd}[sep = small]
 1|1 \arrow[d] & 1|2 \arrow[l] \arrow[d] & \cdots \arrow[l] & 1|n \arrow[l] \arrow[d] & 1|1 \arrow[l] \arrow[d] \\
 2|1 \arrow[d] & 2|2 \arrow[l] \arrow[d] & \cdots \arrow[l] & 2|n \arrow[l] \arrow[d] & 2|1 \arrow[l] \arrow[d] \\
 \vdots \arrow[d] & \vdots \arrow[d] & \ddots & \vdots \arrow[d] & \vdots \arrow[d] \\
 n|1 \arrow[d] & n|2 \arrow[l] \arrow[d] & \cdots \arrow[l] & n|n \arrow[l] \arrow[d] & n|1 \arrow[l] \arrow[d] \\
 1|1 & 1|2 \arrow[l] & \cdots \arrow[l] & 1|n \arrow[l] & 1|1 \arrow[l] \\
\end{tikzcd}
\]
where  we identify the first row with the last row, and the first column with the last column,
modulo the following relations:
\begin{itemize}
\item composition of any two horizontal arrows is 0;
\item composition of any two vertical arrows is 0;
\item all squares commute.
\end{itemize}
Vertical arrows are of the form $\alpha_i\otimes \varepsilon_j$, whereas horizontal arrows are of the form $\varepsilon_i\otimes \alpha^{\mathrm{op}}_j$.
Since we presented $\Lambda_{n} \otimes_{\Bbbk} \Lambda_{n}^{\on{op}}$ using a discrete torus, it is natural that in some arguments we write $\alpha_{k}, \varepsilon_{k}$ for $k > n$, in which case we set $\alpha_{k} := \alpha_{k'}$, for $k' \in \setj{1,\ldots,n}$ such that $k \equiv k' \on{mod} n$, and similarly for $\varepsilon$.

We describe $\Lambda_n$-$\Lambda_n$-bimodules diagrammatically as representation of the above quiver with relations.
For readability we only present the part of the quiver at which the value of a representation is non-zero.
All arrows in the diagrams indicate action via identity operators.

The isomorphism classes of indecomposable $\Lambda_n$-$\Lambda_n$-bimodules form three families: projective-injectives, string bimodules, and band bimodules. 

The band bimodules form a three-parameter family indexed by triples $(j, m, \lambda)$, where $j \in \setj{1,\ldots,n}$, $m$ is a positive integer and $\lambda \in \Bbbk\setminus\setj{0}$ is a non-zero scalar. Similarly to \cite{Jo3}, we do not need to consider band bimodules in our arguments, hence omit a detailed exposition (which can be found for instance in \cite{Jo2}). 

Projective-injectives: for each $i,j\in \{ 1,\ldots ,n\}$, there is an indecomposable bimodule $P_{i|j}=I_{i+1|j-1}$, with the following diagrammatic presentation:
\[
 \begin{tikzcd}
  \Bbbk_{i|j-1} \arrow[d, swap, "\alpha_{i}\cdot"] & \Bbbk_{i|j} \arrow[l, swap, "\cdot \alpha_{j-1}"] \arrow[d, "\alpha_{i} \cdot "] \\
  \Bbbk_{i+|j-1} & \Bbbk_{i+1|j} \arrow[l, "\cdot \alpha_{j-1}"]
 \end{tikzcd}
\]
String bimodules: for each $i,j\in \{1,\ldots ,n\}$ and all nonnegative integers $k$, there are four string bimodules $W_{i|j}^{(k)}$, $S_{i|j}^{(k)}$, $N_{i|j}^{(k)}$ and $M_{i|j}^{(k)}$. Below are examples of small dimensions to illustrate, for more details see \cite{Jo2}.

\[
\resizebox{0.95\textwidth}{!}{$
\begin{array}{ll}
W_{i|j}^{(2)}:\
\begin{tikzcd}[ampersand replacement = \&]
 \Bbbk_{i|j} \arrow[d, swap, "{\alpha_{i}\cdot}"] \\
 \Bbbk_{i+1|j} \& \Bbbk_{i+1|j+1} \arrow[l, swap, "{\cdot \alpha_{j}}"] \arrow[d, swap, "{\alpha_{i+1}\cdot}"] \\
 \& \Bbbk_{i+2|j+1} \& \Bbbk_{i+2| j+2} \arrow[l, swap, "{\cdot \alpha_{j+1}}"]
\end{tikzcd}
& N_{i|j}^{(1)}:\
\begin{tikzcd}[ampersand replacement = \&]
 \Bbbk_{i|j-1}  \& \Bbbk_{i|j} \arrow[d, swap, "{\alpha_{i} \cdot}"] \arrow[l, swap, "{\cdot \alpha_{j-1}}"] \\
 \& \Bbbk_{i+1|j} \& \Bbbk_{i+1|j+1} \arrow[l, swap, "{\cdot \alpha_{j}}"]
\end{tikzcd} \\
S_{i|j}^{(1)}:\
\begin{tikzcd}[ampersand replacement = \&]
 \Bbbk_{i|j} \arrow[d, swap, "\alpha_{i}\cdot "] \\ \Bbbk_{i+1|j} \& \Bbbk_{i+1|j+1} \arrow[l, swap, "\cdot \alpha_{j}"] \arrow[d, swap, "\alpha_{i+1} \cdot"] \\
 \& \Bbbk_{i+2|j+1}
\end{tikzcd}
& M_{i|j}^{(1)}:\
\begin{tikzcd}[ampersand replacement = \&]
 \Bbbk_{i|j-1} \& \Bbbk_{i|j} \arrow[l, swap, "\cdot \alpha_{j-1}"] \arrow[d, swap, "\alpha_{i} \cdot "] \\
 \& \Bbbk_{i+1|j} \& \Bbbk_{i+1|j+1} \arrow[l, swap, "\cdot \alpha_{j}"] \arrow[d, swap, "\alpha_{i+1} \cdot "] \\
 \& \& \Bbbk_{i+2|j+1}
\end{tikzcd}
\end{array}
$}
\]
The index $i|j$ is called the \emph{initial vertex}.  The index $k$ counts the number of \emph{valleys},  i.e. sinks of indegree 2.

\begin{rmk}\label{rmk_index_nm}
Compared to \cite{Jo2}, the indexing of the string bimodules of shapes $N$ and $M$ is shifted. The bimodules $N_{i|j}^{(k)}$ and $M_{i|j}^{(k)}$ would in \cite{Jo2} be called $N_{i|j-1}^{(k)}$ and $M_{i|j-1}^{(k)}$ respectively.
This change in notation gives easier formulas for tensor products in Subsection~\ref{subs_multtable}.
\end{rmk}

Following \cite{MMZ1}, we call an indecomposable $\Lambda_n$-$\Lambda_n$-bimodule \emph{$\Bbbk$-split} if it is of the form $U\otimes_\Bbbk V$, for indecomposable left and right $A$-modules $U$ and $V$. The $\Bbbk$-split $\Lambda_n$-$\Lambda_n$-bimodules are
\begin{itemize}
    \item the projective-injective bimodules $P_{i|j}$,
    \item the simple bimodules $L_{i|j}=W_{i|j}^{(0)}$,
    \item the 2-dimensional bimodules $S_{i|j}^{(0)}$ and $N_{i|j}^{(0)}$.
\end{itemize}

The two-sided cells in the set of isomorphism classes of indecomposable $\Lambda_n$-$\Lambda_n$-bimodules are the following:
\begin{itemize}
\item the cell $\mathcal{J}_{\text{split}}$ consisting of all $\Bbbk$-split bimodules,
\item the cell $\mathcal{J}_{M_0}$ consisting of all bimodules $M_{i|j}^{(0)}$ where $i,j\in \{ 1,\ldots ,n\}$,
\item for each positive integer $k$, the cell $\mathcal{J}_k$ consisting of all string bimodules with exactly $k$ valleys,
\item the cell $\mathcal{J}_{\text{band}}$ consisting of all band bimodules.
\end{itemize}
Moreover, the two-sided cells are linearly ordered as follows:
\begin{align*}
\mathcal{J}_{\text{split}} \geq_J \mathcal{J}_{M_0} \geq_J \mathcal{J}_{1} \geq_J \mathcal{J}_{2} \geq_J \ldots \geq_J \mathcal{J}_{\text{band}} .
\end{align*}
All two-sided cells except $\mathcal{J}_{M_0}$ are idempotent. Moreover, all cells are finite, apart from $\mathcal{J}_{\mathrm{band}}$, which has the same cardinality as the field $\Bbbk$.

\subsection{The main result}
From now on, we fix a positive integer $n$. The main goal of this paper is to prove the following result.

\begin{theorem}\label{thm_main}
Fix a  positive integer $k$. Then the following holds.
\begin{enumerate}[(i)]
\item Any simple transitive birepresentation of $\mathscr{D}_n$ with apex $\mathcal{J}_\mathrm{split}$ is equivalent to a cell birepresentation.
\item Any simple transitive birepresentation of $\mathscr{D}_n$ with apex $\mathcal{J}_k$ has rank between $n$ and $2n$.
\item For each $j=0,\ldots ,n$, there exist exactly $\binom{n}{j}$ pairwise non-equivalent simple transitive birepresentations of $\mathscr{D}_n$ with apex $\mathcal{J}_k$ which have rank $n+j$. Every such birepresentation can be constructed by localizing a cell birepresentation with apex $\mathcal{J}_{k}$ by a suitable $\csym{D}_{n}$-stable collection.
\end{enumerate}
\end{theorem}

\begin{rmk}
\begin{enumerate}
\item In the case $n=1$, Theorem~\ref{thm_main}$(i)-(ii)$ are parts of \cite[Theorem~1]{Jo3}, and Theorem~\ref{thm_main}$(iii)$ is \cite[Conjecture~2]{Jo3}.
\item Generalizing \cite[Theorem~1$(iv)$]{Jo3}, we note that the cell birepresentation corresponding to one of the left cells in $\mathcal{J}_{M_0}$ is a simple transitive birepresentation of $\mathscr{D}_n$ of rank $n$ with apex $\mathcal{J}_1$. In particular, recall that by \cite[Lemma~1]{CM}, the apex of a transitive $2$-representation must be idempotent. The cell $\mathcal{J}_{M_{0}}$ is not idempotent, and thus does not appear in the classification of Theorem~\ref{thm_main}.
\item The minimal $\mathcal{J}$-cell of the bicategory of $\Lambda_{n}\!\on{-}\!\Lambda_{n}$-bimodules, $\mathcal{J}_{\on{band}}$, is even further from the finitary setting than the remaining $\mathcal{J}$-cells. Indeed, the collection of isomorphism classes of indecomposable band bimodules is parametrized by pairs $(n,\lambda)$, for $n \in \mathbb{Z}_{\geq 0}$ and $\lambda \in \Bbbk$. The classification problem for simple transitive birepresentations with apex $\mathcal{J}_{\on{band}}$ goes beyond the scope of this article, and since in the case of finitary apex our approach allows us to focus on the apex only, the band bimodules will not play a great role in our considerations.
\end{enumerate}
\end{rmk}

The remainder of this paper is devoted to the proof of Theorem~\ref{thm_main}. It is structured as follows: Claim~$(i)$ is proved in Subsection~\ref{ss_pf_main_i}. In Section~\ref{s5} we take a closer look at the two-sided cells $\mathcal{J}_k$ and prove Claim~$(ii)$. Finally, Claim~$(iii)$ is proved in Section~\ref{s6}.

\subsection{\texorpdfstring{Proof of Theorem~\ref{thm_main}$(i)$}{Proof of Theorem 8}}\label{ss_pf_main_i}
Mutatis mutandis \cite[Section~3.3]{Jo3}.

\section{\texorpdfstring{The two-sided cell $\mathcal{J}_k$}{The two-sided cell Jk}}\label{s5}
We fix a positive integer $k$. In this section we shall establish some facts about the two-sided cell $\mathcal{J}_k$ in order to better understand the simple transitive birepresentations of $\mathscr{D}_n$ having it as apex.
For readability we sometimes omit the upper index $(k)$ on the elements of $\mathcal{J}_k$, writing $N_{i|j}$ for $N_{i|j}^{(k)}$ and so on.

\subsection{\texorpdfstring{One-sided cells in $\mathscr{J}_k$}{One-sided cells in Jk}}
In \cite{Jo2}, it was shown that every left  cell in  $\mathcal{J}_k$ contains either bimodules of type $W$ and $S$, or bimodules of type $M$ and $N$. Similarly, every right cell contains either bimodules of type $W$ and $N$, or bimodules of  type $S$ and $M$. More precisely, we have the egg-box diagram below. 
The columns of the diagram are the left cells of $\mathcal{J}_{k}$, and the rows are the right cells of $\mathcal{J}_{k}$.
\begin{displaymath}
\begin{array}{|c|c|c|c|c|c|}
\hline
W_{1|1} & \cdots & W_{1|n} & N_{1|1} & \cdots & N_{1|n} \\
\hline
\vdots & \ddots & \vdots & \vdots & \ddots & \vdots \\
\hline
W_{n|1} & \cdots & W_{n|n} & N_{n|1} & \cdots & N_{n|n} \\
\hline
S_{1|1} & \cdots & S_{1|n} & M_{1|1} & \cdots & M_{1|n} \\
\hline
\vdots & \ddots & \vdots & \vdots & \ddots & \vdots \\
\hline
S_{n|1} & \cdots & S_{n|n} & M_{n|1} & \cdots & M_{n|n} \\
\hline
\end{array}
\end{displaymath}
In other words, in each left cell all elements have the same second coordinate of the initial vertex (i.e. the lower index $i|j$). Similarly, in each right cell all elements have the same first coordinate of the initial vertex.

\subsection{Multiplication table}\label{subs_multtable}
With calculations similar to those in \cite{Jo1}, one can show that, modulo direct summands from two-sided  cells strictly $J$-greater than $\mathcal{J}_k$, the multiplication table  of $\mathcal{J}_k$ is given by
\begin{displaymath}
\begin{array}{c|c|c|c|c}
\otimes_{\Lambda_n} & W_{j|l} & S_{j|l} & N_{j|l} & M_{j|l} \\
\hline
W_{i|j} & W_{i|l} & W_{i|l} & N_{i|l} & N_{i|l} \\
\hline
S_{i|j} & S_{i|l} & S_{i|l} & M_{i|l} & M_{i|l}\\
\hline
N_{i|j} & W_{i|l} & W_{i|l} & N_{i|l} & N_{i|l}\\
\hline
M_{i|j} & S_{i|l} & S_{i|l} & M_{i|l} & M_{i|l} \\
\end{array}
\end{displaymath}
for all $i,j,l\in \{ 1,\ldots ,n\}$,
together with
\begin{align*}
    U_{i|j}\otimes_{\Lambda_n} V_{r|s} =0
\end{align*}
for all $U,V\in \{ M,N,W,S \}$ whenever $j\neq r$.
In particular, modulo two-sided cells strictly $J$-greater than $\mathcal{J}_{k}$, the $1$-morphisms $M^{(k)}_{i|i}$, $N^{(k)}_{i|i}$, $W^{(k)}_{i|i}$ and $S^{(k)}_{i|i}$ are all idempotent. Setting
\begin{align*}
F=\bigoplus_{U\in \mathcal{J}_k} U
\end{align*}
yields $F\otimes F\simeq F^{\oplus 4n}$.

Moreover, setting $F_{i|j}=M^{(k)}_{i|j}\oplus N^{(k)}_{i|j}\oplus W^{(k)}_{i|j}\oplus S^{(k)}_{i|j}$ yields 
\begin{align*}
    F_{i|j}\otimes F_{j|l}= F_{i|l}^{\oplus 4}
\end{align*}
and $F_{i|j}\otimes F_{r|s}=0$ for $j\neq r$.

\subsection{Adjoint pairs}
In contrast to the fiat/fiab setting, not every $1$-morphism of $\csym{D}_{n}$ admits a left or right adjoint. We now describe the adjoint pairs in $\mathcal{J}_{k}$.

\begin{prop}\label{prop_adj}
For any non-negative integer $k$ and any $i,j=1,\ldots ,n$, the pair $(S_{i|j}^{(k)}\otimes_{\Lambda_n}-,N_{j|i}^{(k)}\otimes_{\Lambda_n}-)$ is an adjoint pair of endofunctors of $\Lambda_n$-mod.
\end{prop}

\begin{proof}
By \cite[Lemma 13]{MZ2}, it is enough to show that $S_{i|j}^{(k)}$ is projective as a left $\Lambda_n$-module
and that $\Hom _{\Lambda_n\!\on{-mod}}(S_{i|j}^{(k)},\Lambda_n)\simeq N_{j|i}^{(k)}$ as $\Lambda_n$-$\Lambda_n$-bimodules.
Indeed, as left $\Lambda_n$-module $S_{i|j}^{(k)}$ is isomorphic to
\begin{align*}
\Lambda_n\varepsilon_i\oplus \Lambda_n\varepsilon_{i+1}\oplus \ldots \oplus \Lambda_n\varepsilon_{i+k}.
\end{align*}

Consider now the diagrammatic representation of the bimodules $S_{i|j}^{(k)}$ and $N_{j|i}^{(k)}$ in standard bases and with standard $\Lambda_n$-action.
\[
\begin{tikzcd}[sep = small]
 s_{i|j} \arrow[d] \\
 s_{i+1|j} & s_{i+1|j+1} \arrow[l] \arrow[d] \\
 & \ddots & s_{i+k|j+k} \arrow[d] \arrow[l] \\
 & & s_{i+k+1|j+k}
\end{tikzcd}
\]
\[
\begin{tikzcd}[sep = small]
 n_{j|i-1} & n_{j|i} \arrow[l] \arrow[d] \\
 & n_{j+1|i} & n_{j+1|i+1} \arrow[l] \arrow[d] \\
 & & \ddots & n_{j+k|i+k+1} \arrow[l]
\end{tikzcd}
\]

Recall that we have
\begin{align*}
\varepsilon_i s_{i|j}=s_{i|j}=s_{i|j}\varepsilon_j 
\end{align*}
and so on, but also
\begin{align*}
\varepsilon_i s_{i+n|j+n}=s_{i+n|j+n}=s_{i+n|j+n}\varepsilon_j
\end{align*}
and so on.

Now, define linear maps $S_{i|j}^{(k)}\to \Lambda_n$ as follows:
\begin{align*}
f_{i+m|j+m}&: \ 
\begin{cases}
 s_{i+m|j+m}\mapsto \varepsilon_{i+m}\\
 s_{i+m+1|j+m}\mapsto \alpha_{i+m}
\end{cases} ,
  \\
g_{i+m-1|j+m}&:\ s_{i+m|j+m}\mapsto \alpha_{i+m-1}
\end{align*}
with all basis vectors not indicated above mapped to 0.

It is easy to check that
\begin{align*}
\varphi: \Hom_{\Lambda_n\!\on{-mod}}(S_{i|j}^{(k)},\Lambda_n) &\to N_{j|i}^{(k)}\\
f_{i+m|j+m}&\mapsto n_{j+m|i+m}\\
g_{i+m-1|j+m} &\mapsto n_{j+m|i+m-1}
\end{align*}
is  an isomorphism of $\Lambda_n$-$\Lambda_n$-bimodules.
\end{proof}

\begin{cor}\label{ProjTypeN}
Let $\mathbf{M}$ be a simple transitive 2-representation of $\mathscr{D}_n$ with apex $\mathcal{J}_k$, for some $k\geq 1$. Then for all $i,j\in \{1,\ldots ,n\}$, the functor $\overline{\mathbf{M}}(N_{i|j}^{(k)})$ is a projective functor.
\end{cor}

\begin{proof}
By Proposition~\ref{prop_adj}, each $\overline{\mathbf{M}}(N_{i|j}^{(k)})$ is left exact.  Therefore the result  follows by \cite[Theorem~4]{Jo3}.
\end{proof}

\subsection{Action matrices}\label{subs_matrixblocks}
Let $\mathbf{M}$ be a simple transitive 2-representation of $\mathscr{D}_n$ with apex $\mathcal{J}_k$. In this section we completely classify all possible action matrices for $\mathbf{M}U$, where $U$ is in the $\mathcal{J}_k$.
This is done by block decomposition of the action matrix of $\mathbf{M}F$ (with $F$ as in the previous subsection), and reduction of diagonal blocks to the case $n=1$, which was considered in \cite{Jo3}.
As a by-product of this computation, we prove Theorem~\ref{thm_main}$(ii)$.
Thereafter we consider off-diagonal blocks.

To simplify notation, for a $1$-morphism $G$, we write $[G]$ instead of $[\mathbf{M}(G)]$.

\subsubsection{Block decomposition and diagonal blocks}
Using definitions and the results from the previous section, we observe the following:
\begin{itemize}
\item Since $\mathbf{M}$ is simple transitive, each entry of $[F]$ is a positive integer. By the above, it satisfies $[F]^2=4n[F]$, so its trace is $4n$.
\item As a simple transitive 2-representation does not annihilate any element of its apex, each $[U_{i|j}^{(k)}]$ is nonzero.
\item For $i=1,\ldots ,n$ and $U\in \{M,N,S,W\}$, each entry of $[U_{i|i}^{(k)}]$ is a nonnegative integer, and the matrix itself is idempotent. Thus these matrices have diagonal elements 0 and/or 1, and trace equal to rank.
\item For $i\neq j$, each $[U_{i|j}^{(k)}]$ has nonnegative integer entries and  squares to 0, and therefore has zero diagonal.
\end{itemize}
As the number of 1-morphisms of the form $U_{i|i}^{(k)}$ is exactly $4n$, we conclude that the diagonal of each $[U_{i|i}^{(k)}]$ contains exactly one entry equal to 1, and all the remaining diagonal entries are zeros.

If $[U_{i|i}^{(k)}]$ and $[V_{j|j}^{(k)}]$ have their unique 1 on the diagonal in the same position, then $[U_{i|i}^{(k)}][V_{j|j}^{(k)}]=[U_{i|i}^{(k)}\otimes_{\Lambda_{n}} V_{j|j}^{(k)}]\neq 0$. This implies $i=j$.

We can now choose an ordering of the indecomposable objects in $\mathbf{M}(\mathtt{i})$ such that the first elements of $\mathrm{diag}[F]$ are the nonzero elements of $\mathrm{diag}[F_{1|1}]$, the next elements of $\mathrm{diag}[F]$ are the nonzero elements of $\mathrm{diag}[F_{2|2}]$, and so on.

Let $n_1$ be the number of nonzero elements in $\mathrm{diag}[F_{1|1}]$. Assume that $[F_{i|j}]$ has a nonzero element in one of the first $n_1$ rows. Then $[F_{1|1}][F_{i|j}]\neq 0$, implying that $i=1$. Similarly, if $[F_{i|j}]$ has a nonzero element in on of the first $n_1$ columns, then $[F_{i|j}][F_{1|1}]\neq 0$, implying $j=1$. The same arguments can be applied to the rows and columns where $[F_{2|2}],\ldots ,[F_{n|n}]$ have their nonzero diagonal entries.

By the above argument, we can write $[F]$ as a block matrix
\begin{align*}
[F]=
\begin{bmatrix}
A_{1|1} & \cdots & A_{1|n}\\
\vdots & \ddots & \vdots \\
A_{n|1} & \cdots & A_{n|n}
\end{bmatrix}
\end{align*}
where the block $A_{i|j}$ is the nonzero part of $[F_{i|j}]$. For example, $[F_{1|2}]$ is the block matrix
\begin{align*}
\begin{bmatrix}
0 & A_{1|2} & 0\\
0 & 0 & 0
\end{bmatrix} .
\end{align*}

Then each of the blocks $A_{i|j}$ is a matrix with only positive integer entries. 
We have
\begin{align*}
    [F_{i|j}][F_{r|s}]=4\delta_{jr}[F_{i|s}],
\end{align*}
where $\delta$ denotes the Kronecker delta.
In particular each $A_{i|i}$ must satisfy the relation $A_{i|i}A_{i|i}=4A_{i|i}$.
Moreover, let $A_{i|j}(U)$ be the part of $[U_{i|j}^{(k)}]$ contained in the block $A_{i|j}$. 
Using the above calculations and Proposition~\ref{prop_adj} in the case $i=j$, we see that  $A_{i|i}(M)$, $A_{i|i}(N)$, $A_{i|i}(S)$ and $A_{i|i}(W)$ satisfy all the relations of $[M_k]$, $[N_k]$, $[S_k]$ and $[W_k]$ from \cite[Section~4]{Jo3}.
Consequently, we have, for each $i$, either
\begin{align*}
A_{i|i}(M)=A_{i|i}(N)=A_{i|i}(S)=A_{i|i}(W)=[1],
\end{align*}
or
\begin{align*}
A_{i|i}(N)=A_{i|i}(W)=&\begin{bmatrix}
1&1\\ 0& 0
\end{bmatrix}\\
A_{i|i}(S)=A_{i|i}(M)=&\begin{bmatrix}
0&0\\
1&1
\end{bmatrix},
\end{align*}
up to natural action of the symmetric group $S_{2}$ on $\on{Mat}_{\mathbb{Z}_{\geq 0}}(2\times 2)$.

By definition, the matrix $[F]$ is of size $\mathrm{rank} \mathbf{M}\times \mathrm{rank} \mathbf{M}$.
Therefore we can note that
\begin{align}\label{ThmII}
\mathrm{rank}\mathbf{M} = 2n- \left| \{ i\mid A_{i|i}(N)=[1]\}\right| = n+ \left| \left\lbrace i\mid A_{i|i}(N)=\begin{bmatrix}
1&1\\ 0& 0
\end{bmatrix}\right\rbrace \right| .
\end{align}
This proves Theorem~\ref{thm_main}$(ii)$.

\subsubsection{Off-diagonal blocks}
Assume now that $n\geq 2$ and fix $i,j\in \{1,\ldots ,n\}$. We shall now describe $A_{i|j}$ and $A_{j|i}$ using our previous observations about $A_{i|i}$ and $A_{j|j}$.

If $A_{i|i}(U)=A_{j|j}(U)=[1]$ for all $U$, then both $A_{i|j}$ and $A_{j|i}$ are $1\times 1$-matrices. Since
\begin{align*}
A_{i|j}(U)A_{j|i}(U)=A_{i|i}(U)=[1]
\end{align*}
for all $U$, we have $A_{i|j}(U)=A_{j|i}(U)=[1]$ for all $U$ as well.

Assume now that $A_{i|i}(U)=[1]$ for all $U$, and that
\begin{align*}
A_{i|i}(N)=A_{i|i}(W)=&\begin{bmatrix}
1&1\\ 0& 0
\end{bmatrix}\\
A_{i|i}(S)=A_{i|i}(M)=&\begin{bmatrix}
0&0\\
1&1
\end{bmatrix}.
\end{align*}
Then $A_{i|j}$ is a $1\times 2$-matrix and $A_{j|i}$ is a $2\times 1$-matrix. If we write $A_{i|j}(N)=[a\ b]$ and
$A_{j|i}(N)=\begin{bmatrix}
c\\ d
\end{bmatrix}$, then
\begin{align*}
A_{i|j}(N)=A_{i|j}(N)A_{j|j}(N)= [a\ a].
\end{align*}
Similarly
\begin{align*}
A_{j|i}(N)=A_{j|j}(N)A_{j|i}(N) = \begin{bmatrix} c+d\\ 0\end{bmatrix},
\end{align*}
implying $d=0$. Now
\begin{align*}
[1]= A_{i|i}(N)= A_{i|j}(N)A_{j|i}(N) =
\begin{bmatrix}
a&a
\end{bmatrix}
\begin{bmatrix}
c\\ 0
\end{bmatrix}
= [ac],
\end{align*}
so that $a=c=1$.

Applying similar arguments to $W$, $S$, $M$, we find
\begin{align*}
A_{i|j}(N)=A_{i|j}(W)=A_{i|j}(S)=A_{i|j}(M)=
\begin{bmatrix}
1 & 1
\end{bmatrix}&, \\
A_{j|i}(N)=A_{j|i}(W) =
\begin{bmatrix}
1\\ 0
\end{bmatrix}&, \\
A_{j|i}(S)=A_{j|i}(M)=
\begin{bmatrix}
0\\ 1
\end{bmatrix} &.
\end{align*}

Finally, assume
\begin{align*}
A_{i|i}(N)=A_{i|i}(W)=A_{j|j}(N)=A_{j|j}(W)=&\begin{bmatrix}
1&1\\ 0& 0
\end{bmatrix}\\
A_{i|i}(S)=A_{i|i}(M)=A_{j|j}(S)=A_{j|j}(M)=&\begin{bmatrix}
0&0\\
1&1
\end{bmatrix}.
\end{align*}
If we write
$A_{i|j}(N)=
\begin{bmatrix}
a&b\\ c&d
\end{bmatrix}$, then
\begin{align*}
\begin{bmatrix}
a&b\\ c&d
\end{bmatrix}
=
A_{i|j}(N)
=
A_{i|i}(N)A_{i|j}(N)
=\begin{bmatrix}
a+c & b+d \\ 0&0
\end{bmatrix},
\end{align*}
so that $c=d=0$.
By the same argument, $A_{j|i}(N)=\begin{bmatrix}
a' & b'\\ 0&0
\end{bmatrix}$
for some $a',b'$. Then
\begin{align*}
\begin{bmatrix}
1&1\\ 0& 0
\end{bmatrix} =A_{i|i}(N) = A_{i|j}(N)A_{j|i}(N) = 
\begin{bmatrix}
aa' & ab' \\ 0&0
\end{bmatrix},
\end{align*}
so that $a=a'=b'=1$.
Using
\begin{align*}
\begin{bmatrix}
1&1\\ 0& 0
\end{bmatrix} =A_{j|j}(N) = A_{j|i}(N)A_{i|j}(N)
\end{align*}
yields also $b=0$, so that
\begin{align*}
A_{i|j}(N)=A_{j|i}(N)=\begin{bmatrix}
1&1\\ 0& 0
\end{bmatrix}.
\end{align*}
By similar arguments, we conclude that
\begin{align*}
A_{i|j}(N)=A_{i|j}(W)=A_{j|i}(N)=A_{j|i}(W)=&\begin{bmatrix}
1&1\\ 0& 0
\end{bmatrix}\\
A_{i|j}(S)=A_{i|j}(M)=A_{j|i}(S)=A_{j|i}(M)=&\begin{bmatrix}
0&0\\
1&1
\end{bmatrix}.
\end{align*}

\section{\texorpdfstring{Simple transitive birepresentations of $\mathscr{D}_n$}{Simple transitive birepresentations of Dn}}\label{s6}
In this section we prove  Theorem~\ref{thm_main}$(iii)$.

 Let $\mathbf{M}$ be a simple transitive birepresentation of $\csym{D}_{n}$. Let $B$ be a basic algebra such that $\mathbf{M}(\mathtt{i}) \simeq B\!\on{-proj}$. In the arguments that follow, we identify these two categories, in particular, we will write $Be \in \mathbf{M}(\mathtt{i})$, for an idempotent $e \in B$. Further, the indecomposable $1$-morphisms of the form $U_{i|j}$, for $U \in \setj{M,N,W,S}$, are always assumed to lie in the apex of $\mathbf{M}$.
 
 From the earlier matrix calculations, we know that, for any $i,j \in \setj{1,\ldots, n}$ and $U \in \setj{M,N,W,S}$, the action matrix $[U_{i|j}]$ of the indecomposable $1$-morphism $U_{i|j}$ has a unique non-zero row. 
 Hence there is an indecomposable object $Be_{i_{U}}$ of $\mathbf{M}(\mathtt{i})$ such that the essential image of $\mathbf{M}U_{i|j}$ lies in $\on{add}\setj{Be_{i_{U}}}$. 
 From our matrix analysis it also follows that, for all $i$, we have $\on{add}\setj{Be_{i_{N}}} = \on{add}\setj{Be_{i_{W}}}$ and $\on{add}\setj{Be_{i_{M}}} = \on{add}\setj{Be_{i_{S}}}$, so $Be_{i_{N}} \simeq Be_{i_{W}}$ and $Be_{i_{M}} \simeq Be_{i_{S}}$. 
 Similarly, we have $Be_{i_{S}} \simeq Be_{i_{N}}$ if and only if the block matrix $A_{i|i}$ is $1 \times 1$. In this case, we choose a representative $Be_{i_{S,N}}$ of the resulting common isomorphism class. We will often indicate this case by writing $e_{i_{S}} = e_{i_{N}}$, and it's negation by $e_{i_{S}} \neq e_{i_{N}}$. In the former case, we denote the common idempotent by $e_{i_{S,N}}$.

 Finally, we have also shown that, for any $j$, all the rows in the matrix 
 \[
 \left[\bigoplus_{i=1}^{n} (N_{i|j} \oplus S_{i|j})\right]
 \]
 are non-zero, so any indecomposable object in $\mathbf{M}(\mathtt{i})$ is isomorphic to an object of the form $Be_{i_{N}}$ or $Be_{i_{S}}$, for some $i$. Thus, the collection
\[
 \bigcup_{e_{i_{S}} \neq e_{i_{N}}}\!\setj{Be_{i_{S}},Be_{i_{N}}} \quad \cup\!\bigcup_{e_{i_{S}} = e_{i_{N}}}\!\setj{Be_{i_{S,N}}}
\]
is a complete and irredundant collection of isomorphism classes of indecomposable objects of $\mathbf{M}(\mathtt{i})$. Immediately, we obtain 
\[
\on{rank}\mathbf{M} = 2n - \left|\setj{i \; | \; e_{i_{S}} = e_{i_{N}} }\right| = n + \left|\setj{i \; | \; e_{i_{S}} \neq e_{i_{N}} }\right| ,
\]
which can be viewed as a rewriting of \eqref{ThmII}.

\begin{proposition}\label{MNindecProj}
 Let $U \in \setj{M_{i|j},N_{i|j}}$. Then $\overline{\mathbf{M}}U$ is an indecomposable projective functor.
\end{proposition}

\begin{proof}
 From Corollary~\ref{ProjTypeN} we know that $\overline{\mathbf{M}}N_{i|j}$ is a projective functor. From the multiplication table in Subsection \ref{subs_multtable}, together with \cite[Lemma 8]{MZ1}, we conclude that also $\overline{\mathbf{M}}M_{i|j}$ is a projective functor.
 
 If the action matrix of $\mathbf{M}U$ has a unique non-zero entry, it must equal $1$ and it immediately follows that $\overline{\mathbf{M}}U$ is indecomposable.
 
 Otherwise, in the unique non-zero row there are two non-zero columns, each with unique non-zero entry equal to $1$, associated to the indices $j_{S},j_{N}$. For $\overline{\mathbf{M}}U$ not to be indecomposable, we must then have 
 \[
  \overline{\mathbf{M}}U \simeq (Be_{k} \otimes_{\Bbbk} e_{j_{S}}B) \oplus (Be_{k} \otimes_{\Bbbk} e_{j_{N}}B),
 \]
  with $k = i_{S}$ if $U = M_{i|j}$, and $k = i_{N}$ if $U = N_{i|j}$. In this case, the action matrices of both summands have a unique non-zero entry which equals $1$, which implies that both $e_{j_{N}}B$ and $e_{j_{S}}B$ are simple. But, combining the description of action matrices given in Subsection \ref{subs_matrixblocks} with the adjunction $(S_{j|i},N_{i|j})$ yields
  \[
  \begin{aligned}
   &\on{Hom}_{\mathbf{M}(\mathtt{i})}(Be_{j_{S}},Be_{j_{N}}) \simeq \on{Hom}_{\mathbf{M}(\mathtt{i})}(\mathbf{M}S_{j|i}(Be_{i_{N}}),Be_{j_{N}}) \\ 
   &\simeq \on{Hom}_{\mathbf{M}(\mathtt{i})}(Be_{i_{N}},\mathbf{M}N_{i|j}(Be_{j_{N}})) \simeq \on{Hom}_{\mathbf{M}(\mathtt{i})}(Be_{i_{N}},Be_{i_{N}}) \neq 0
  \end{aligned}
  \]
  so $\on{Hom}_{\mathbf{M}(\mathtt{i})}(Be_{j_{S}},Be_{j_{N}}) \neq 0$, showing that $Be_{j_{N}}$ is not simple.
\end{proof}

\begin{proposition}\label{ActionDescribed}
 For $i,j \in \setj{1,\ldots,n}$, we have
 \[
  \begin{aligned}
   &\overline{\mathbf{M}}N_{i|j} \simeq Be_{i_{N}} \otimes_{\Bbbk} e_{j_{S}}B \otimes_{B} - \\
   &\overline{\mathbf{M}}S_{i|j} \simeq Be_{i_{S}} \otimes_{\Bbbk} \left(Be_{j_{N}}\right)^{\ast} \otimes_{B} - \\
   &\overline{\mathbf{M}}M_{i|j} \simeq Be_{i_{S}} \otimes_{\Bbbk} e_{j_{S}}B \otimes_{B} - \\
   &\overline{\mathbf{M}}W_{i|j} \simeq Be_{i_{N}} \otimes_{\Bbbk} \left(Be_{j_{N}}\right)^{\ast} \otimes_{B} -
  \end{aligned}
 \]
 where $(-)^{\ast} = \on{Hom}_{\Bbbk}(-,\Bbbk)$. 
\end{proposition}

\begin{proof}
 Since $\overline{\mathbf{M}}N_{i|j}$ is indecomposable and its only non-zero row has index $i_{N}$, it is clear that there is some index $l$ such that 
 \[
  \overline{\mathbf{M}}N_{i|j} \simeq Be_{i_{N}} \otimes_{\Bbbk} e_{l}B  \otimes_{B} - .
 \]
 Let $L_{j_{S}}$ be the simple object associated to index $j_{S}$. Using the adjunction $(S_{j|i},N_{i|j})$, we have
 \[
   \on{Hom}_{\overline{\mathbf{M}}(\mathtt{i})}(\overline{\mathbf{M}}S_{j|i}(B), L_{j_{S}}) \simeq \on{Hom}_{\overline{\mathbf{M}}(\mathtt{i})}\left(B, \overline{\mathbf{M}}N_{i|j}(L_{j_{S}})\right)
 \]
 where the left-hand side is non-zero, since from the action matrix of $\mathbf{M}S_{j|i}$ we have $\mathbf{M}S_{j|i}(B) \simeq Be_{j_{S}}^{\oplus 2}$, or $\mathbf{M}S_{j|i}(B) \simeq Be_{j_{S}}$, in case $e_{j_{S}} = e_{j_{N}}$. In both cases we conclude that $\on{Hom}_{\overline{\mathbf{M}}(\mathtt{i})}(B,\overline{\mathbf{M}}N_{i|j}(L_{j_{S}}))$ is non-zero, which shows $l = j_{S}$. 
 
 From Proposition~\ref{MNindecProj} we know that $\overline{\mathbf{M}}M_{i|j}$ is an indecomposable projective functor, and so again there is some index $l$ such that 
 \[
  \overline{\mathbf{M}}M_{i|j} \simeq Be_{i_{S}} \otimes_{\Bbbk} e_{l}B \otimes_{B} - .
 \]
 The claim $l = j_{S}$ follows from the isomorphism $\overline{\mathbf{M}}M_{i|j} \simeq \overline{\mathbf{M}}M_{i|j} \circ \overline{\mathbf{M}}N_{j|j}$. 
 
 The claim about $\overline{\mathbf{M}}S_{i|j}$ follows directly from the adjunction $(N_{j|i},S_{i|j})$ together with the adjunction
 \[
  \left( Be_{i_{S}} \otimes_{\Bbbk} \left(Be_{j_{N}}\right)^{\ast} \otimes_{B} -,  Be_{j_{N}} \otimes_{\Bbbk} e_{i_{S}}B \otimes_{B} -   \right),
 \]
 which can be shown in the same manner as the adjunction in \cite[Lemma 45]{MM1}. 
 
 Finally, to see that $\overline{\mathbf{M}}W_{i|j} \simeq Be_{i_{N}} \otimes_{\Bbbk} \left(Be_{j_{N}}\right)^{\ast} \otimes_{B} -$, let $T_{i|j}$ be a $B$-$B$-bimodule such that $\overline{\mathbf{M}}W_{i|j} \simeq T_{i|j} \otimes_{B} -$.
 
 Since $\overline{\mathbf{M}}W_{i|j} \simeq \overline{\mathbf{M}}N_{i|i} \circ \overline{\mathbf{M}}W_{i|j}$, we have
 \[
  T_{i|j} \simeq Be_{i_{N}} \otimes_{\Bbbk} \left(e_{j_{S}}B \otimes_{B} T_{i|j}\right).
 \]
 Using this, write $T_{i|j} \simeq Be_{i_{N}} \otimes_{\Bbbk} R_{i|j}$. From the action matrix of $\mathbf{M}W_{i|j}$ we see that the composition multiplicity $[R_{i|j} : L_{j_{S}}]$ is $1$. Further, we have 
 \[
  \overline{\mathbf{M}}W_{i|j} \simeq \overline{\mathbf{M}}W_{i|j} \circ \overline{\mathbf{M}}S_{j|j},
 \]
 so
 \[
  Be_{i_{N}} \otimes_{\Bbbk} R_{i|j} \simeq Be_{i_{N}} \otimes_{\Bbbk} \left( R_{i|j} \otimes_{B} Be_{j_{S}} \right) \otimes_{\Bbbk} \left(Be_{j_{N}}\right)^{\ast} \simeq Be_{i_{N}} \otimes_{\Bbbk} \left(Be_{j_{N}}\right)^{\ast}
 \]
 where the last isomorphism is due to $\on{dim}\left( R_{i|j} \otimes_{B} Be_{j_{S}} \right) = [R_{i|j} : L_{j|s}] = 1$.
\end{proof}

\begin{proposition}\label{ProjInjA2}
 $e_{j_{S}}B \simeq \left(Be_{j_{N}}\right)^{\ast}$. In case $e_{j_{S}} = e_{j_{N}}$, this module is simple. Otherwise, it is of length two, with socle $L_{j_{N}}$ and top $L_{j_{S}}$.
 
 As a consequence, $\overline{\mathbf{M}}N_{i|j} \simeq \overline{\mathbf{M}}W_{i|j}$ and $\overline{\mathbf{M}}S_{i|j} \simeq \overline{\mathbf{M}}M_{i|j}$.
\end{proposition}

\begin{proof}
 In the case $e_{j_{S}} = e_{j_{N}}$, the unique non-zero column of the matrix $[N_{i|j}]$ is indexed by $j_{S,N}$. This column has a unique non-zero entry equal to $1$, from which it follows that, for a simple $B$-module $L$, we have $[e_{j_{S,N}}B:L] \neq 0$ if and only if $L \simeq L_{j_{S,N}}$, for which we have $[e_{j_{S,N}}B:L]=1$.
 
 In case $e_{j_{S}} \neq e_{j_{N}}$, we find that the only non-zero columns of $[N_{i|j}]$ are those indexed by $j_{S}$ and $j_{N}$, both admitting a unique non-zero entry equal to $1$. This shows that the composition factors of $e_{j_{S}}B$ are $L_{j_{S}}$ and $L_{j_{N}}$, both with multiplicity one.
 
 The same arguments apply to $(Be_{j_{N}})^{\ast}$, if one replaces the matrix $[N_{i|j}]$ by $[S_{i|j}]$ in the above considerations.
 
 Finally, from the above observations about the matrix $[N_{i|j}]$ we see that the module $e_{j_{S}}B \otimes_{B} Be_{j_{N}} \simeq e_{j_{S}}Be_{j_{N}}$ is one-dimensional, so there is at most one arrow $j_{N} \rightarrow j_{S}$ in the quiver of $B$. This, together with the first part of the result, proves that the two modules are isomorphic.
 
 The last part of the statement is an immediate consequence of Proposition~\ref{ActionDescribed}.
\end{proof}

\begin{proposition}\label{CartanMatrix}
 For $U,V \in \setj{S,N}$ and $i,j \in \setj{1,\ldots, n}$, we have:
 \[
 \on{dim}e_{i_{U}}Be_{j_{V}} = \on{dim}\on{Hom}_{\mathbf{M}(\mathtt{i})}(Be_{i_{U}},Be_{j_{V}}) = 
  \begin{cases}
   1 \text{, if } i=j \text{ and } U = V; \\
   1 \text{, if } i=j, \; U=S \text{ and } V=N; \\
   0 \text{, otherwise.}
  \end{cases}
 \]
\end{proposition}

\begin{proof}
 The first two cases are an immediate consequence of Proposition~\ref{ProjInjA2}. 
  
  If there is any morphism $\alpha: Be_{i_{U}} \rightarrow Be_{j_{V}}$ outside the identity morphisms and those of the form $Be_{i_{S}} \rightarrow Be_{i_{N}}$, then again from our earlier description of $e_{k_{S}}B$ we see that $e_{k_{S}}B \otimes_{B} \alpha = 0$, for $k=1,\ldots,n$. 
  
  In view of the characterization given in Proposition~\ref{ActionDescribed}, this shows that the $\csym{D}_{n}$-stable ideal of $\mathbf{M}(\mathtt{i})$ generated by $\alpha$ does not coincide with all of $\mathbf{M}(\mathtt{i})$, and hence is a proper ideal, contradicting $\mathbf{M}$ being simple transitive.
\end{proof}

\begin{corollary}\label{LabelledQuiver}
 The quiver of $B$, together with the labelling we use above, is given as follows:
 \begin{itemize}
  \item it consists of $n$ connected components, labelled by $i \in \setj{1,\ldots,n}$, each of type $A_{1} = \bullet$ or $A_{2} = \bullet \rightarrow \bullet$;
  \item the connected component indexed by $i$ is of type $A_{1}$ if and only if $e_{i_{S}} = e_{i_{N}}$. Its unique vertex is labelled by $i_{S,N}$;
  \item the connected component indexed by $i$ is of type $A_{2}$ if and only if $e_{i_{S}} \neq e_{i_{N}}$. It is labelled as $i_{N} \rightarrow i_{S}$.
 \end{itemize}
\end{corollary}

\begin{proof}
 The result is an immediate consequence of Propositions~\ref{ProjInjA2} and \ref{CartanMatrix}.
\end{proof}

The labelled quiver described in Corollary~\ref{LabelledQuiver} should be compared with the quiver underlying the basic algebra $C$ satisfying $C\!\on{-proj} \simeq \mathbf{C}(\mathtt{i})$, where $\mathbf{C}$ is a cell birepresentation with apex $\mathcal{J}_{k}$. The choice of $j$ below does not matter as the different cell birepresentations are equivalent, hence we suppress it from our notation, writing $\mathbf{C}$ rather than $\mathbf{C}_{j}$. The quiver underlying $C$ is of the following form:
  \[
   \begin{tikzcd}[sep = small]
    M_{1|j} \arrow[d, "\alpha_{1}"] & M_{2|j} \arrow[d, "\alpha_{2}"] & \cdots & M_{n|j} \arrow[d, "\alpha_{n}"] \\
    N_{1|j} & N_{2|j} & \cdots & N_{n|j}
   \end{tikzcd}.
  \]
Its labelling illustrates the fact that the indecomposable objects of $\mathbf{C}(\mathtt{i})$ are indecomposable $1$-morphisms of $\mathcal{L}$, and the arrows correspond to the bimodule epimorphisms $M_{i|j} \rightarrow N_{i|j}$.

Consider the strong transformation $\Theta_{L_{j_{S}}}: \mathbf{P}_{\mathtt{i}} \rightarrow \overline{\mathbf{M}}$, uniquely determined by the assignment $\mathbb{1}_{\mathtt{i}} \mapsto L_{j_{S}}$. 
For a $1$-morphism $U$ of $\csym{D}_{n}$, we have 
\[
\Theta_{L_{j_{S}}}(U) \simeq \overline{\mathbf{M}}U(L_{j_{S}}),
\]
and so Proposition~\ref{ActionDescribed} implies that the codomain of $\Theta_{L_{j_{S}}}$ can be restricted to $\mathbf{M}$, since $\mathbf{M}$ is canonically embedded as the category of projective objects of $\overline{\mathbf{M}}$.  

Similarly to \cite[Theorem~6.2]{Zi2} and \cite[Proposition~9]{MM5}, we conclude that $\Theta_{L_{j_{S}}}$ induces a locally faithful strong transformation $\Sigma: \mathbf{C} \rightarrow \mathbf{M}$ which satisfies 
\[
\Sigma(U_{i|j}) \simeq \overline{\mathbf{M}}U_{i|j}(L_{j_{S}}).
\]

\begin{lemma}\label{FactorLoc}
 For any $i \in \setj{1,\ldots, n}$, the morphism $\mathbf{M}\alpha_{i}$ is an isomorphism if and only if $e_{i_{S}} = e_{i_{N}}$.
\end{lemma}

\begin{proof}
If $e_{i_{S}} = e_{i_{N}}$, then we have $\mathbf{M}M_{i|j} \simeq \mathbf{M}N_{i|j}$, and so $\mathbf{M}\alpha$ is a morphism between isomorphic objects: $\mathbf{M}\alpha \in \on{Hom}_{\mathbf{M}(\mathtt{i})}\big(\mathbf{M}M_{i|j}(L_{j_{S}}), \mathbf{M}N_{i|j}(L_{j_{S}})\big)$. The morphism $\mathbf{M}\alpha_{i}$ is non-zero since $\Sigma$ is faithful. Thus, $\mathbf{M}\alpha_{i}$ is an isomorphism, since the endomorphism algebra of $\mathbf{M}N_{i|j}(L_{j_{S}}) \simeq Be_{i_{N}}$ is simple.

If $e_{i_{S}} \neq e_{i_{N}}$, then $\mathbf{M}M_{i|j}(L_{j_{S}}) \not\simeq \mathbf{M}N_{i|j}(L_{j_{S}})$, which completes the proof.
\end{proof}

\begin{theorem}\label{Classification}
 Let $I = \setj{i^{1},\ldots,i^{k}} \subseteq \setj{1,\ldots,n}$ be the collection of indices for which we have $i_{S} = i_{N}$. Let $\euler{S} = \setj{\alpha_{i} \; |\; i \in I}$. The birepresentation $\mathbf{M}$ is equivalent to the localized birepresentation $\mathbf{C}\big[\euler{S}^{-1}\big]$.
\end{theorem}

\begin{proof}

From $\overline{\mathbf{M}}N_{i|j}(L_{j_{S}}) \simeq Be_{i_{N}}$ and $\overline{\mathbf{M}}M_{i|j}(L_{j_{S}}) \simeq Be_{i_{S}}$, we see that $\Sigma$ is essentially surjective. Further, we have previously observed that $\Sigma$ is faithful. 
 
 From Lemma \ref{FactorLoc} and the universal property of $\mathbf{C}\big[ \euler{S}^{-1}\big]$, we see that there is a strong transformation $\Sigma[\euler{S}^{-1}]$ such that 
 \[
  \begin{tikzcd}
   \mathbf{C} \arrow[r, "\Upsilon"] \arrow[d, "\Sigma"] & \mathbf{C}[\euler{S}^{-1}] \arrow[dl, "{\Sigma[\euler{S}^{-1}]}"] \\
    \mathbf{M} 
  \end{tikzcd}
 \]
 commutes up to invertible modification. Since $\Sigma$ is essentially surjective, so is $\Sigma[\euler{S}^{-1}]$.
 
 We now claim that, given indecomposable objects $X,Y \in \mathbf{C}[\euler{S}^{-1}]$, the hom-spaces $\on{Hom}_{\mathbf{C}[\euler{S}^{-1}]}(X,Y)$ and $\on{Hom}_{\mathbf{C}[\euler{S}^{-1}]}\big(\Sigma[\euler{S}^{-1}](X),\Sigma[\euler{S}^{-1}](Y)\big)$ are equidimensional. To see this, we recall that we have explicitly described the quiver for the category $\mathbf{C}[\euler{S}^{-1}](\mathtt{i})$ in Example~\ref{LocalizeA2} and that it is the same as the quiver for $\mathbf{M}(\mathtt{i})$. 
 Following the description from Example~\ref{LocalizeA2}, we also find that, on the level of quivers, $\Upsilon$ is given by ''contracting'' the $A_{2}$-components of the quiver of $\mathbf{C}(\mathtt{i})$ labelled by indices $i$ such that $e_{i_{S}} = e_{i_{N}}$. 
 Using Proposition~\ref{CartanMatrix} and Lemma \ref{FactorLoc}, the same conclusion can be made about $\Sigma$,  which proves the claim about equidimensionality.
 
 Next we show that $\Sigma[\euler{S}^{-1}]$ is faithful. To do so, we show that it is such on the level of indecomposable objects. Since all the $\on{Hom}$-spaces in $\mathbf{C}[\euler{S}^{-1}]$ are at most $1$-dimensional, it suffices that we show that $\Sigma[\euler{S}^{-1}](\alpha_{i}) \neq 0$, for all $i \in \setj{1,\ldots,n}$. Recall that $\Sigma$ is faithful, so
 \[
  \Sigma[\euler{S}^{-1}](\alpha_{i}) = \Sigma[\euler{S}^{-1}]\circ \Upsilon (\alpha_{i}) = \Sigma(\alpha_{i}) \neq 0
 \]
 which proves that $\Sigma[\euler{S}^{-1}]$ is faithful.
 
 We conclude that, for any objects $X,Y\in \mathbf{C}[\euler{S}^{-1}]$, the map
 \[
  \on{Hom}_{\mathbf{C}[\euler{S}^{-1}]}(X,Y) \xrightarrow{\Sigma[\euler{S}^{-1}]_{X,Y}} \on{Hom}_{\mathbf{C}[\euler{S}^{-1}]}(\Sigma[\euler{S}^{-1}](X),\Sigma[\euler{S}^{-1}](Y))
 \]
 is an injective $\Bbbk$-linear map between equidimensional spaces, hence a bijection. Thus $\Sigma[\euler{S}^{-1}]$ is essentially surjective, full and faithful. The result follows.
\end{proof}

\begin{proposition}
 Given $I,I' \subseteq \setj{1,\ldots,n}$, let $\euler{S} = \setj{\alpha_{i} \; | \; i \in I}$ and $\euler{S}' = \setj{\alpha_{i} \; | \; i \in I'}$. The birepresentations $\mathbf{C}[\euler{S}^{-1}], \mathbf{C}[\euler{S}'^{-1}]$ are equivalent if and only if $I = I'$.
\end{proposition}

\begin{proof}
 Clearly we only need to prove that $\mathbf{C}[\euler{S}^{-1}] \simeq \mathbf{C}[\euler{S}'^{-1}]$ implies $I = I'$. From Proposition~\ref{ActionDescribed} we find that $\overline{\mathbf{C}[\euler{S}^{-1}]}M_{i|j} \simeq \overline{\mathbf{C}[\euler{S}^{-1}]}N_{i|j}$ if and only if $i \in I$.
 
 If $\mathbf{C}[\euler{S}^{-1}] \simeq \mathbf{C}[\euler{S}'^{-1}]$, then 
 \[
\overline{\mathbf{C}[\euler{S}^{-1}]}M_{i|j} \simeq \overline{\mathbf{C}[\euler{S}^{-1}]}N_{i|j} \text{ implies }\overline{\mathbf{C}[\euler{S}'^{-1}]}M_{i|j} \simeq \overline{\mathbf{C}[\euler{S}'^{-1}]}N_{i|j},
 \]
 so
 $
  i \in I \text{ implies } i \in I'
 $
 and the result follows.
\end{proof}

\begin{corollary}\label{FinalClaim}
 The map  
 \begin{equation}
  \begin{aligned}
 \Big\{ 
  \begin{aligned}
  \text{Subsets of } \setj{1,\ldots,n}
  \end{aligned}
&\Big\}
 \longrightarrow 
  \Bigg\{ 
  \begin{aligned}
  \text{Simple} &\text{ transitive birepresentations} \\ &\text{of } \csym{D}_{n} \text{ with apex } \mathcal{J}_{k}
  \end{aligned}
\Bigg\}/\simeq
 \\
 &I \longrightarrow  \mathbf{C}[\euler{S}^{-1}]
 \end{aligned}
  \end{equation}
is a bijection. Further, $\on{rank}\mathbf{C}[\euler{S}^{-1}] = 2n - |I|$.
\end{corollary}

Theorem~\ref{thm_main}$(iii)$ follows from Theorem~\ref{Classification} and Corollary~\ref{FinalClaim}.

\vspace{5mm}

\noindent
Helena Jonsson, Department of Mathematics, Uppsala University, Box. 480,
SE-75106, Uppsala, SWEDEN, email: {\tt helena.jonsson\symbol{64}math.uu.se}

Mateusz Stroi{\' n}ski, Department of Mathematics, Uppsala University, Box. 480,
SE-75106, Uppsala, SWEDEN, email: {\tt mateusz.stroinski\symbol{64}math.uu.se}

\end{document}